\documentclass[12pt]{amsart}
\usepackage{amsmath,amsthm,amsfonts,amssymb,latexsym, mathrsfs}

\usepackage{hyperref}

\usepackage{enumerate}
\usepackage{color}
\usepackage[dvipsnames]{xcolor}

\usepackage[shortlabels]{enumitem}

\begin{document}

\theoremstyle{plain}

\newtheorem{thm}{Theorem}[section]

\newtheorem{lem}[thm]{Lemma}
\newtheorem{Problem B}[thm]{Problem B}

\newtheorem{pro}[thm]{Proposition}
\newtheorem{conj}[thm]{Conjecture}
\newtheorem{cor}[thm]{Corollary}
\newtheorem{que}[thm]{Question}
\newtheorem{prob}[thm]{Problem}
\newtheorem{rem}[thm]{Remark}
\newtheorem{defi}[thm]{Definition}
\newtheorem{cond}[thm]{Condition}

\newtheorem*{thmA}{Theorem A}
\newtheorem*{thmB}{Theorem B}
\newtheorem*{corB}{Corollary B}
\newtheorem*{thmC}{Theorem C}
\newtheorem*{thmD}{Theorem D}
\newtheorem*{thmE}{Theorem E}
 
\newtheorem*{thmAcl}{Main Theorem$^{*}$}
\newtheorem*{thmBcl}{Theorem B$^{*}$}
\newcommand{\dd}{\mathrm{d}}

\newtheorem{thml}{Theorem}
\renewcommand*{\thethml}{\Alph{thml}}   
\newtheorem{conjl}[thml]{Conjecture}
\newtheorem{condl}[thml]{Condition}

\newcommand{\Maxn}{\operatorname{Max_{\textbf{N}}}}
\newcommand{\Syl}{\operatorname{Syl}}
\newcommand{\Lin}{\operatorname{Lin}}
\newcommand{\U}{\mathbf{U}}
\newcommand{\R}{\mathbf{R}}
\newcommand{\dl}{\operatorname{dl}}
\newcommand{\Con}{\operatorname{Con}}
\newcommand{\cl}{\operatorname{cl}}
\newcommand{\Stab}{\operatorname{Stab}}
\newcommand{\Aut}{\operatorname{Aut}}
\newcommand{\Ker}{\operatorname{Ker}}
\newcommand{\InnDiag}{\operatorname{InnDiag}}
\newcommand{\fl}{\operatorname{fl}}
\newcommand{\Irr}{\operatorname{Irr}}
\newcommand{\FF}{\mathbb{F}}
\newcommand{\EE}{\mathbb{E}}
\newcommand{\normal}{\trianglelefteq}
\newcommand{\sn}{\normal\normal}
\newcommand{\Bl}{\mathrm{Bl}}
\newcommand{\NN}{\mathbb{N}}
\newcommand{\N}{\mathbf{N}}
\newcommand{\bfC}{\mathbf{C}}
\newcommand{\bfO}{\mathbf{O}}
\newcommand{\bfF}{\mathbf{F}}
\def\GGG{{\mathcal G}}
\def\HHH{{\mathcal H}}
\def\HH{{\mathcal H}}
\def\irra#1#2{{\rm Irr}_{#1}(#2)}

\renewcommand{\labelenumi}{\upshape (\roman{enumi})}

\newcommand{\PSL}{\operatorname{PSL}}
\newcommand{\PSU}{\operatorname{PSU}}
\newcommand{\alt}{\operatorname{Alt}}

\providecommand{\V}{\mathrm{V}}
\providecommand{\E}{\mathrm{E}}
\providecommand{\ir}{\mathrm{Irm_{rv}}}
\providecommand{\Irrr}{\mathrm{Irm_{rv}}}
\providecommand{\re}{\mathrm{Re}}

\numberwithin{equation}{section}
\def\irrp#1{{\rm Irr}_{p'}(#1)}

\def\ibrrp#1{{\rm IBr}_{\Bbb R, p'}(#1)}
\def\C{{\mathbb C}}
\def\Q{{\mathbb Q}}
\def\irr#1{{\rm Irr}(#1)}
\def\irrp#1{{\rm Irr}_{p^\prime}(#1)}
\def\irrq#1{{\rm Irr}_{q^\prime}(#1)}
\def \c#1{{\cal #1}}
\def \aut#1{{\rm Aut}(#1)}
\def\cent#1#2{{\bf C}_{#1}(#2)}
\def\norm#1#2{{\bf N}_{#1}(#2)}
\def\zent#1{{\bf Z}(#1)}
\def\syl#1#2{{\rm Syl}_#1(#2)}
\def\normal{\triangleleft\,}
\def\oh#1#2{{\bf O}_{#1}(#2)}
\def\Oh#1#2{{\bf O}^{#1}(#2)}
\def\det#1{{\rm det}(#1)}
\def\gal#1{{\rm Gal}(#1)}
\def\ker#1{{\rm ker}(#1)}
\def\normalm#1#2{{\bf N}_{#1}(#2)}
\def\alt#1{{\rm Alt}(#1)}
\def\iitem#1{\goodbreak\par\noindent{\bf #1}}
   \def \mod#1{\, {\rm mod} \, #1 \, }
\def\sbs{\subseteq}

\def\gc{{\bf GC}}
\def\ngc{{non-{\bf GC}}}
\def\ngcs{{non-{\bf GC}$^*$}}
\newcommand{\notd}{{\!\not{|}}}

\newcommand{\Z}{\mathbf{Z}}
\newcommand{\Out}{{\mathrm {Out}}}
\newcommand{\Mult}{{\mathrm {Mult}}}
\newcommand{\Inn}{{\mathrm {Inn}}}
\newcommand{\IBR}{{\mathrm {IBr}}}
\newcommand{\IBRL}{{\mathrm {IBr}}_{\ell}}
\newcommand{\IBRP}{{\mathrm {IBr}}_{p}}
\newcommand{\cd}{\mathrm{cd}}
\newcommand{\ord}{{\mathrm {ord}}}
\def\id{\mathop{\mathrm{ id}}\nolimits}
\renewcommand{\Im}{{\mathrm {Im}}}
\newcommand{\Ind}{{\mathrm {Ind}}}
\newcommand{\diag}{{\mathrm {diag}}}
\newcommand{\soc}{{\mathrm {soc}}}
\newcommand{\End}{{\mathrm {End}}}
\newcommand{\sol}{{\mathrm {sol}}}
\newcommand{\Hom}{{\mathrm {Hom}}}
\newcommand{\Mor}{{\mathrm {Mor}}}
\newcommand{\Mat}{{\mathrm {Mat}}}
\def\rank{\mathop{\mathrm{ rank}}\nolimits}
\newcommand{\Tr}{{\mathrm {Tr}}}
\newcommand{\tr}{{\mathrm {tr}}}
\newcommand{\Gal}{{\rm Gal}}
\newcommand{\Spec}{{\mathrm {Spec}}}
\newcommand{\ad}{{\mathrm {ad}}}
\newcommand{\Sym}{{\mathrm {Sym}}}
\newcommand{\Char}{{\mathrm {Char}}}
\newcommand{\pr}{{\mathrm {pr}}}
\newcommand{\rad}{{\mathrm {rad}}}
\newcommand{\abel}{{\mathrm {abel}}}
\newcommand{\PGL}{{\mathrm {PGL}}}
\newcommand{\PCSp}{{\mathrm {PCSp}}}
\newcommand{\PGU}{{\mathrm {PGU}}}
\newcommand{\codim}{{\mathrm {codim}}}
\newcommand{\ind}{{\mathrm {ind}}}
\newcommand{\Res}{{\mathrm {Res}}}
\newcommand{\Lie}{{\mathrm {Lie}}}
\newcommand{\Ext}{{\mathrm {Ext}}}
\newcommand{\Alt}{{\mathrm {Alt}}}
\newcommand{\AAA}{{\sf A}}
\newcommand{\SSS}{{\sf S}}
\newcommand{\DDD}{{\sf D}}
\newcommand{\QQQ}{{\sf Q}}
\newcommand{\CCC}{{\sf C}}
\newcommand{\SL}{{\mathrm {SL}}}
\newcommand{\Sp}{{\mathrm {Sp}}}
\newcommand{\PSp}{{\mathrm {PSp}}}
\newcommand{\SU}{{\mathrm {SU}}}
\newcommand{\GL}{{\mathrm {GL}}}
\newcommand{\GU}{{\mathrm {GU}}}
\newcommand{\Spin}{{\mathrm {Spin}}}
\newcommand{\CC}{{\mathbb C}}
\newcommand{\CB}{{\mathbf C}}
\newcommand{\RR}{{\mathbb R}}
\newcommand{\QQ}{{\mathbb Q}}
\newcommand{\ZZ}{{\mathbb Z}}
\newcommand{\bfN}{{\mathbf N}}
\newcommand{\bfZ}{{\mathbf Z}}
\newcommand{\PP}{{\mathbb P}}
\newcommand{\cG}{{\mathcal G}}
\newcommand{\OO}{\mathcal O}
\newcommand{\cH}{{\mathcal H}}
\newcommand{\cQ}{{\mathcal Q}}
\newcommand{\GA}{{\mathfrak G}}
\newcommand{\cT}{{\mathcal T}}
\newcommand{\cL}{{\mathcal L}}
\newcommand{\IBr}{\mathrm{IBr}}
\newcommand{\cS}{{\mathcal S}}
\newcommand{\cR}{{\mathcal R}}
\newcommand{\GCD}{\GC^{*}}
\newcommand{\TCD}{\TC^{*}}
\newcommand{\FD}{F^{*}}
\newcommand{\GD}{G^{*}}
\newcommand{\HD}{H^{*}}
\newcommand{\GCF}{\GC^{F}}
\newcommand{\TCF}{\TC^{F}}
\newcommand{\PCF}{\PC^{F}}
\newcommand{\GCDF}{(\GC^{*})^{F^{*}}}
\newcommand{\RGTT}{R^{\GC}_{\TC}(\theta)}
\newcommand{\RGTA}{R^{\GC}_{\TC}(1)}
\newcommand{\Om}{\Omega}
\newcommand{\eps}{\epsilon}
\newcommand{\varep}{\varepsilon}
\newcommand{\al}{\alpha}
\newcommand{\chis}{\chi_{s}}
\newcommand{\sigmad}{\sigma^{*}}
\newcommand{\PA}{\boldsymbol{\alpha}}
\newcommand{\gam}{\gamma}
\newcommand{\lam}{\lambda}
\newcommand{\la}{\langle}
\newcommand{\genf}{F^*}
\newcommand{\ra}{\rangle}
\newcommand{\hs}{\hat{s}}
\newcommand{\htt}{\hat{t}}
\newcommand{\tG}{\hat G}
\newcommand{\St}{\mathsf {St}}
\newcommand{\bfs}{\boldsymbol{s}}
\newcommand{\bfl}{\boldsymbol{\lambda}}
\newcommand{\tn}{\hspace{0.5mm}^{t}\hspace*{-0.2mm}}
\newcommand{\ta}{\hspace{0.5mm}^{2}\hspace*{-0.2mm}}
\newcommand{\tb}{\hspace{0.5mm}^{3}\hspace*{-0.2mm}}
\def\skipa{\vspace{-1.5mm} & \vspace{-1.5mm} & \vspace{-1.5mm}\\}
\newcommand{\tw}[1]{{}^#1\!}
\newcommand{\Irrg}[1]{\Irr_{p',\sigma}(#1)}
\renewcommand{\mod}{\bmod \,}

\newcommand{\bG}{\mathbf{G}}
\newcommand{\bbG}{\mathbb{G}}
\newcommand{\bg}{\mathbf}
\newcommand{\type}{\operatorname}
\newcommand{\wt}{\widetilde}
\newcommand{\sym}{\mathfrak{S}}

\marginparsep-0.5cm

\newcommand{\mandicomment}{\textcolor{teal}}

\renewcommand{\thefootnote}{\fnsymbol{footnote}}
\footnotesep6.5pt
\title{On almost $p$-rational characters in principal blocks}
\author[A. Mar\'oti, J. M. Mart\'inez, A. A. Schaeffer Fry, C. Vallejo]{Attila Mar\' oti, J. Miquel Mart\'inez, A. A. Schaeffer Fry and Carolina Vallejo}
\subjclass{$20$C$15$ ($20$C$20$)}

\thanks{The first-named author was supported by the National Research, Development
and Innovation Office (NKFIH) Grant No.~K138596, No.~K132951 and Grant
No.~K138828. The second-named author is funded by the national project PRIN 2022- 2022PSTWLB - Group Theory and Applications - CUP B53D23009410006. The third-named author gratefully acknowledges support from the National Science
Foundation, Award No. DMS-2100912, and her former institution, Metropolitan
State University of Denver, which holds the award and allows her to serve as PI. The fourth-named author, as part of the GNSAGA, is grateful for the support of the Istituto Nazionale di
Alta Matematica (INdAM). She also acknowledges support from the Rita Levi-Montalcini 2019 programme. Part of this work was carried out while the second-named author visited the University of Florence during the fall of 2023, funded by a {\em Borsa Ferran Sunyer i Balaguer} 2023 from {\em l'Institut d'Estudis Catalans}. He thanks {\em l'Institut} for the support and the {\em Dipartamento di Matematica e Informatica `Ulisse Dini'} for the warm hospitality. The authors thank Gunter Malle and the anonymous reviewers for a thorough reading of earlier versions of this manuscript and for suggestions that helped improve the exposition.}

\maketitle

\begin{abstract}
Let $p$ be a prime. In this paper we provide a lower bound for the number of almost $p$-rational characters of degree coprime to $p$ in the principal $p$-block of a finite group of order divisible by $p$. We further describe the $p$-local structure of the groups for which the above-mentioned bound is sharp. 
\end{abstract}

\section{Introduction}

Let $G$ be a finite group and let $p$ be a prime dividing $|G|$. For a positive integer $k$, let $\xi_k$ be a primitive $k$th root of unity. We denote by $\mathbb{Q}_k=\mathbb{Q}(\xi_k)$ the $k$th cyclotomic extension of $\mathbb{Q}$.
Let $\sigma \in {\rm Gal}(\QQ^{\mathrm{ab}}/\QQ)$ be the Galois automorphism that sends $p$-power roots of unity to their $1+p$ power and fixes roots of unity of order coprime to $p$ (by abusing notation, we use $\sigma$ to denote the restriction to any cyclotomic extension of $\mathbb{Q}$). 
Although the definition of $\sigma$ depends on the prime $p$, it will be clear from context throughout the article. 
Following \cite{Hun-Mal-Mar22}, we say that a character $\chi$ is almost $p$-rational if all the values of $\chi$ lie in some cyclotomic extension $\QQ_{pm}$ where $(p,m)=1$. 
Write $|G|=p^a n$ with $(p,n)=1$ and $a$ is a positive integer. If $p$ is odd, then ${\rm Gal}(\QQ_{|G|}/\QQ_{pn})=\langle \sigma \rangle$. In particular, a character $\chi$ of $G$ is almost $p$-rational if, and only if, $\chi$ is $\sigma$-invariant. 
If $p=2$, then $\QQ_{2n}=\QQ_n$, so almost $2$-rational characters are $2$-rational. As ${\rm Gal}(\QQ_{|G|}/\QQ_{n})$ is not generally cyclic, we cannot characterize $2$-rational characters as those invariant under the action of $\langle \sigma \rangle$ (for instance, dihedral groups of order $2^k>8$ have irreducible characters that are $\sigma$-invariant but not 2-rational). However, when dealing with irreducible characters $\chi$ of odd degree, it was first noticed in \cite[Theorem 2.5]{Val23} that
$\chi$ is $2$-rational if, and only if, $\chi$ is $\sigma$-invariant (or equivalently $\langle \sigma\rangle$-fixed). 

Let $G$ be a finite group and let $p$ be a prime. For a $p$-block $B$ of $G$, let $k(B)$ denote the number of complex irreducible characters in $B$. In \cite{Het-Kul00}, H\'ethelyi and K\"ulshammer conjecture that if $k(B)>1$ then $$k(B)\geq2\sqrt{p-1}.$$ This conjecture remains open even for solvable groups.
Recently, Hung and the third-named author of this article have shown that the H\' ethelyi--K\"ulshammer conjecture holds for principal blocks in \cite{Hun-Sch23}. 
Our aim in this article is to show that a much stronger bound is true for principal blocks. More precisely, we show that $2\sqrt{p-1}$ is also a lower bound for the
number of almost $p$-rational characters of degree coprime to $p$ lying in the principal $p$-block of groups of order divisible by $p$. From the discussion above, this means we are interested in studying 
$$k_{0, \sigma}(B_0(G)),$$
 the size of the set $\displaystyle \Irr_{p', \sigma}(B_0(G))$
of irreducible characters of degree coprime to $p$ in the principal $p$-block of $G$ fixed by the action of $\langle \sigma \rangle$.
We care to remark that the numbers $k_{0, \sigma}(B_0(G))$ have been recently shown to influence the number of generators of the Sylow $p$-subgroups in \cite{Nav-Tie19, Riz-Sch-Val20, Val23}. We denote by $\Phi(G)$ the Frattini subgroup of a group $G$ and by $\lceil x \rceil \in \ZZ$ the ceiling of a real number $x$.
 Our main result is the following.

\begin{thml}\label{thm:maintheorem}
Let $G$ be a finite group, $p$ a prime and $P\in\Syl_p(G)$. If $P>1$, then 
$$\displaystyle k_{0, \sigma}(B_0(G))\geq  \lceil 2\sqrt{p-1} \rceil\, $$
with equality if, and only if, the following conditions hold:
\begin{enumerate}
\item $\lceil 2\sqrt{p-1}\rceil=e + \frac {p-1} e$ for some divisor $e$ of $p-1$;
\item $P$ is cyclic; and 
\item the local quotient group $\norm G P/\Phi(P)\oh{p'}{\norm G P}$ is isomorphic to one of the following Frobenius groups $\displaystyle\mathsf{C}_p\rtimes \mathsf{C}_{e}$ or $\displaystyle \mathsf{C}_p\rtimes \mathsf{C}_{\frac{p-1} e}$.
\end{enumerate}
\end{thml}

The conclusion of Theorem A can be proven to be a consequence of the Alperin--McKay--Navarro conjecture \cite[Conjecture B]{Nav04}.
In this sense, Theorem A provides new evidence of the validity of this conjecture, which as of the writing of this article is quite far from being proved. (For example, it has not yet  been reduced to a problem on simple groups.) We give further details in Section \ref{sec:AMN}, where we also speculate about the possibility of extending the inequality in Theorem \ref{thm:maintheorem} to arbitrary blocks.
 
 \smallskip
 
 Given any finite group $G$ of order divisible by a prime $p$, it was proven in \cite{Mar16}, by the first-named author of this article, that the number $k(G)$ 
of conjugacy classes of $G$ (or equivalently the number of irreducible complex characters of $G$) is bounded below by $2\sqrt{p-1}$. 
 Moreover, the results of  \cite{Mar16} show that
the bound $k(G)=2\sqrt{p-1}$ is attained only if $a=\sqrt{p-1}\in \mathbb N$ and $G$ is isomorphic to the Frobenius group ${\sf C}_p\rtimes {\sf C}_a$. 
Our proof of the equality in Theorem \ref{thm:maintheorem} ultimately depends on the 
analysis of the equality $k(G)=\lceil 2 \sqrt{p-1}\rceil$ in groups $G$ with a nontrivial normal elementary abelian Sylow $p$-subgroup. 
We carry out this analysis, improving upon \cite[Theorem 1.1]{Mar16}, in Section \ref{sec:conjugacyclasses}.
Additionally, our proof of Theorem \ref{thm:maintheorem} depends on the following statements regarding mostly finite simple groups.

\begin{thml}\label{thm:simple groups}
Let $S$ be a finite nonabelian simple group of order divisible by a prime $p\geq 5$. Let  $A=\Aut(S)$. 
\begin{enumerate}
\item There is some $X$-invariant $1_S\neq \theta\in\Irrg{B_0(S)}$, where $X/S\in\Syl_p(A/S)$.
\item If $S$ has non-cyclic Sylow $p$-subgroups and $S\leq T\leq A$ with $p\nmid|T:S|$ then
\begin{enumerate}
\item $k_{0,\sigma}(B_0(T))>\lceil 2\sqrt{p-1} \rceil$,
\item the number of $T$-orbits on $\Irrg{B_0(S)}$ is at least $2(p-1)^{1/4}$.
\end{enumerate}
\item Suppose that $S$ has cyclic Sylow $p$-subgroups. Write $N=S^t$ with $t\geq 2$ and suppose that $N\normal H$ with $p\nmid|H:N|$. 
Then $$k_{0,\sigma}(B_0(H))=k_{0,\sigma}(B_0(\norm H P))>\lceil2\sqrt{p-1} \rceil$$ where $P\in \Syl_{p}({H})$.
\end{enumerate}
\end{thml}

We note that Theorem \ref{thm:simple groups}(i) implies that there is a $\sigma$-invariant extension of $\theta$ to $X$, using \cite[Cors. 6.2 and 6.4]{Nav18}. With this in mind, we prove a stronger statement in the case that $S$ is a  simple group of Lie type in non-defining characteristic in Lemma \ref{lem:existsprincseriesunip} below, which might be useful for related future applications.

Notice also that Theorem \ref{thm:simple groups}(iii) states that the principal $p$-block of a group $H$ as described in its statement satisfies \cite[Conjecture D]{Isa-Nav02}. In \cite[Theorem 4.4]{Hun-Mal-Mar22}, the authors prove a weaker version of the statement of Theorem B(iii) by using the Classification of Finite Simple Groups (CFSG, for short) when proving that $2\sqrt{p-1}$
is a lower bound for the number of almost $p$-rational irreducible characters of degree coprime to $p$ in a group of order divisible by $p$. 
Remarkably, we prove Theorem \ref{thm:simple groups}(iii) without invoking the CFSG in Section \ref{sec:cyclicSylow}.

\smallskip

This paper is organized as follows. In Section \ref{sec:conjugacyclasses} we show that if $G$ is a finite group of order divisible by $p$ and $k(G)=\lceil 2\sqrt{p-1}\rceil$ then $G$ has cyclic Sylow $p$-subgroups. Section \ref{sec:princblockfolklore} contains known results on principal blocks. We prove Theorem \ref{thm:simple groups}(i) and (ii) in Section \ref{sec:simples} and Theorem \ref{thm:simple groups}(iii) in Section \ref{sec:cyclicSylow}. We complete the proof of Theorem \ref{thm:maintheorem} in  Section \ref{sec:main}. In Section \ref{sec:final} we discuss some related problems.

\section{An improved bound on the number of conjugacy classes}\label{sec:conjugacyclasses}
The aim of this section is to characterize the finite groups $G$ with an elementary abelian nontrivial Sylow $p$-subgroup 
such that $k(G)=\lceil 2\sqrt{p-1}\rceil$.  
 By \cite{Mar16}, the inequality $k(G)\geq \lceil 2\sqrt{p-1}\rceil$ always holds.
 We will show that $k(G)=\lceil 2\sqrt{p-1}\rceil$ implies that $|G|_p=p$, where $n_p$ is the $p$-part of the number $n$.
 In particular, $k(G)=a+\frac{p-1}a$ for some divisor $a$ of $p-1$, and $G$ is isomorphic to one of the Frobenius groups
 ${\sf C}_p\rtimes {\sf C}_a$ or  $\displaystyle {\sf C}_p\rtimes {\sf C}_{\frac{p-1}a}$.

 \subsection{Statement of results}

Let $G$ be a finite group and $p$ a prime. Let $k(G)$ be the number of conjugacy classes of $G$. For a normal subgroup $V$ in $G$, let $n(G,V)$ denote the number of orbits of $G$ on $V$ (acting by conjugation).
	
The purpose of this section is to prove the following. 	
	
\begin{thm}
\label{rank2}
If $G$ is a finite group having an elementary abelian minimal normal subgroup $V$ of $p$-rank at least $2$ and $|G/V|$ is not divisible by the prime $p$, then $k(G) \geq 2 \sqrt{p-1} +1$. 
\end{thm} 	
	
We make a few remarks on Theorem \ref{rank2}. 

In order to prove Theorem \ref{rank2}, we may assume that $G$ is not nilpotent. For if $G$ is nilpotent, then $k(G) \geq |V| \geq p^{2} \geq 2 \sqrt{p-1} +1$. In the paper \cite{Ver-San07} and its prequels, all non-nilpotent finite groups are classified with at most $14$ conjugacy classes. By going through these lists of groups, we see that no group $G$ is a counterexample to Theorem \ref{rank2} with $k(G) \leq 14$. This means that in Theorem \ref{rank2} we can assume that $2\sqrt{p-1}+1 \geq 15$. In other words, $p \geq 53$. 	
	
H\'ethelyi and K\"ulshammer \cite{Het-Kul03} proved that if $X$ is a finite solvable group whose order is divisible by the square of a prime $p$ then $k(X) \geq (49p + 1)/60$. Thus if $G$ is solvable (and $p \geq 53$), then $k(G) \geq (49p + 1)/60 \geq 2 \sqrt{p-1} + 1$. We may assume in Theorem \ref{rank2} that $G$ is not solvable.	
	
Since $k(G) \geq k(G/V) + n(G,V) - 1$ by the Clifford-Gallagher formula (see the comment after \cite[Theorem 1.10a]{Sch07} and also \cite[Proposition 3.1b]{Sch07}), it is sufficient to show that $k(G/V) + n(G,V) \geq 2 \sqrt{p-1} + 2$. Since $k(G/V) \geq k(G/\cent{G}{V})$ and $n(G,V) = n(G/\cent{G}{V},V)$, we may assume that $G/V$ acts faithfully on $V$, that is, $V$ is a faithful and irreducible $G/V$-module. 

We find that Theorem \ref{rank2} would be a consequence of the following statement (where the new $G$ is $G/V$ in Theorem \ref{rank2}).

\begin{thm}
\label{module}
Let $V$ be an irreducible and faithful $FG$-module for some finite non-solvable group $G$ and finite field $F$ of characteristic $p$ at least $53$. If $p$ does not divide $|G|$, then $k(G) + n(G,V) -1 \geq 2\sqrt{p-1} +1$.  
\end{thm} 	
 
\subsection{Preliminaries} 
 
The size of the field $F$ will be denoted by $q$, the dimension of $V$ over $F$ by $n$, and the center of $\GL(n,q)$ by $Z$. 

We will use the following trivial observation throughout the paper. 

\begin{lem}
	\label{trivial}
	With the notation and assumptions above, $|V|/|G| \leq n(G,V)$.
\end{lem} 
 
Suppose that $G$ transitively permutes a set $\{ V_{1}, \ldots , V_{t} \}$ of proper subspaces of $V$ with $t$ an integer with $1 \leq t \leq n$ as large as possible with the property that $V = V_{1} \oplus \cdots \oplus V_{t}$. Let $B$ be the kernel of this action of $G$ on the set of subspaces. Note that $G/B$ is a transitive permutation group of degree $t$. The subgroup $B$ is isomorphic to a subdirect product of $t$ copies of a finite group $T$. In other words $B$ is isomorphic to a subgroup of $T_{1} \times \cdots \times T_{t}$ where for each $i$ with $1 \leq i \leq t$ the vector space $V_{i}$ is a faithful $T_{i}$-module and $T_{i} \cong T$. Let $H_{1}$ be the stabilizer of $V_{1}$ in $G$. Let $k$ be the number of orbits of $H_{1}$ on $V_1$. 

The following is \cite[Lemma 2.6]{Fou69}.

\begin{lem}
	\label{l99}
	With the above notation and assumptions, $$\max \{ t+1, k  \} \leq \binom{t+k-1}{k-1} \leq n(G,V).$$
\end{lem}

We will use the following consequence of a result of Seager \cite[Theorem 1]{Sea87}.

\begin{pro}   
	\label{Seager}
	Let $W$ be a faithful primitive $FH$-module for a finite solvable group $H$ not contained in $\mathrm{\Gamma L}(1,p^{m})$ where $F$ is a field of prime order $p \geq 53$ and $|W| = p^{m}$. Then $p^{m/2}/12m < n(H,W)$. 
\end{pro}

The following is \cite[Proposition 4]{Pal-Pyb98}. 

\begin{pro}
	\label{PP}
	Let $X$ be any subgroup of the symmetric group $S_m$ whose order is coprime to a prime $p$. If $m > 1$ then $|X| < p^{m-1}$. 
\end{pro}

The proof of the next lemma is identical to the proof of \cite[Lemma 3.5]{Mar16}. 

\begin{lem}
	\label{abelian}
	If $G$ has an abelian subgroup of index at most ${|V|}^{1/2}/(2 \sqrt{p-1}+1)$ and $n(G,V) \leq 2\sqrt{p-1}+1$, then $k(G) \geq 2\sqrt{p-1} +1$.
\end{lem}

\subsection{The class $\mathcal{C}_{q}$}

Our first aim in proving Theorem \ref{module} is to describe (as much as possible) the possibilities for $G$ and $V$ with the condition that $n(G,V) < 2 \sqrt{q-1} + 1$ where $q$ is the size of the underlying field $F$. For this we need to introduce a class of pairs $(G,V)$ which we denote by $\mathcal{C}_{q}$.  

Next we define a class of pairs $(G,V)$ where $V$ is an $FG$-module. Let $W$ be a not necessarily faithful $QH$-module for some finite field extension $Q$ of $F$ and some finite group $H$, and such that the characteristic of $Q$ does not divide $|H|$. We write $\mathrm{Stab}_{Q_{1}}^{Q}(H,W)$ for the class of pairs $(H_{1},W_{1})$ with the property that $W_{1}$ is a $Q_{1}H_{1}$-module with $F \leq Q_{1} \leq Q$ where $W_{1}$ is just $W$ viewed as a $Q_{1}$-vector space and $H_{1}$ is some group with the following property. If $\varphi : H_{1} \rightarrow \GL(W_{1})$ and $\psi : H \rightarrow \GL(W)$ denote the natural, not necessarily injective homomorphisms, then $\varphi(H_{1}) \cap \GL(W) = \psi(H)$. We write $\mathrm{Ind}(H,W)$ for the class of pairs $(H_{1},W_{1})$ with the property that $W_{1} = \mathrm{Ind}_{H}^{H_{1}}(W)$ for some group $H_{1}$ with $H \leq H_{1}$.

\begin{defi}
Inheriting the notation from the previous paragraph, we define $\mathcal{C}_{q}$ to be the class of all pairs $(G,V)$ with the property that $V$ is a finite, faithful and irreducible $FG$-module such that the characteristic of $F$ does not divide $|G|$ and that $(G,V)$ can be obtained by repeated applications of $\mathrm{Stab}_{Q_{1}}^{Q_{2}}$ and $\mathrm{Ind}$ starting with $(H,W)$ where $W$ is a $1$-dimensional $QH$-module with $Q$ a field extension of $F$. 
\end{defi}

The main results of this section are Lemmas \ref{order} and \ref{new}. 

\begin{lem}
	\label{order}
	Let $(G,V) \in \mathcal{C}_{q}$ and assume $n(G,V) < 2\sqrt{q-1} + 1$. If $q \geq 53$, then $|G| < {|V|}^{3/2}$. 
\end{lem} 

Lemma \ref{order} is a slight variation of Lemma 4.1 of \cite{Mar16}. The proof of Lemma \ref{order} may be obtained from the proof of Lemma 4.1 of \cite{Mar16} by minor changes in places (and using Lemma \ref{l99} and Proposition \ref{PP}).

After minor changes to the proof of Lemma 4.2 of \cite{Mar16}, the condition $p \geq 59$ may be relaxed to $p \geq 53$ in the statement and the two occurrences of $2 \sqrt{p-1}$ may be changed to $2 \sqrt{p-1} + 1$. We obtain the following.

\begin{lem}
	\label{new}
	Let $(G,V) \in \mathcal{C}_{q}$ and assume $n(G,V) < 2\sqrt{q-1} + 1$. If $p \geq 53$, then at least one of the following holds.
	\begin{enumerate}
		\item $G$ has an abelian subgroup of index at most ${|V|}^{1/2}/(2\sqrt{p-1}+1)$.
		
		\item $|F| = p$, the module $V$ is induced from a $1$-dimensional module, and $G$ has a factor group isomorphic to $\mathfrak{A}_{n}$ or $\mathfrak{S}_{n}$ where $n = \dim_{F} (V)$. In this case we either have $n=1$, or $15 \leq n \leq 180$ and $p < 8192$.
	\end{enumerate}
\end{lem}

\subsection{Some absolutely irreducible representations} 

The following is Proposition 5.1 of \cite{Mar16} with two occurrences of $59$ changed to $53$ and one occurrence of $2 \sqrt{q-1}$ changed to $2 \sqrt{q-1} + 1$. 

\begin{pro}
	\label{quasisimple}
	Suppose that $p$ is a prime at least $53$. Let $H$ be a finite subgroup of $\GL(n,q)$ with generalized Fitting subgroup a quasisimple group where $q$ is a power of $p$. Put $G = Z \circ H$ where $Z$ is the multiplicative group of $F$. Furthermore suppose that $V$ is an absolutely irreducible $FT$-module for every non-central normal subgroup $T$ of $G$. Suppose also that $|G|$ is not divisible by $p$. Then $n(G,V) \geq 2\sqrt{q-1} + 1$ unless possibly if $n=2$, $q$ is in the range $53 \leq q \leq 14 389$, it is congruent to $\pm 1$ modulo $10$, and $G = Z \circ 2.\mathfrak{A}_{5}$. 
\end{pro} 

Suppose that the group $G$ has a unique normal subgroup $R$ which is minimal subject to being non-central. Suppose that $R$ is an $r$-group of symplectic type for some prime $r$ (this is an $r$-group all of whose characteristic abelian subgroups are cyclic). Suppose that $V$ is an absolutely irreducible $FR$-module. Let $|R/Z(R)| = r^{2a}$ for some positive integer $a$. Then the dimension of the module is $n = r^{a}$. Suppose that $Z \leq G$. The group $G/(R Z)$ can be considered as a subgroup of the symplectic group $\Sp_{2a}{(r)}$. As always, we assume that $q \geq p \geq 53$.

The following is Proposition 5.2 of \cite{Mar16} with the obvious changes ($59$ became $53$ and $2 \sqrt{q-1}$ is changed to $2 \sqrt{q-1}+1$) together with a relaxed condition, the bound $2297$ is now $2351$.  

\begin{pro}
	\label{symplectic}
	Suppose that $V$ and $G$ satisfy the assumptions of the previous paragraph. If $n(G,V) < 2\sqrt{q-1}$, then $n=2$, $59 \leq q=p \leq 2351$, and $|G/Z| \leq 24$.  
\end{pro}

\subsection{Bounding $n(G,V)$}

The following is Theorem 6.1 of \cite{Mar16} with the following changes. The occurrence of $59$ is changed to $53$, the lower bound $2 \sqrt{q-1}$ is changed to $2 \sqrt{q-1} + 1$, and $2297$ is changed to $2351$.

\begin{thm}
	\label{orbit}
	Let $V$ be a finite, faithful, coprime and irreducible $FG$-module. Suppose that the characteristic $p$ of the underlying field $F$ is at least $53$. Put $q = |F|$ and $|V| = q^{n}$. Let the center of $\GL(n,q)$ be $Z$. Then $n(G,V) \geq 2\sqrt{q-1}+1$ unless possibly if one of the following cases holds. 
	\begin{enumerate}
		
		\item $(G,V) \in \mathcal{C}_{q}$;
		
		\item $V = \mathrm{Ind}_{H}^{G}(W)$ for some $2$-dimensional $FH$-module $W$ where $H$ is as $G$ in Proposition \ref{quasisimple} or Proposition \ref{symplectic} satisfying one of the following.
		
		\begin{enumerate}
			\item $53 \leq q \leq 14 389$, $q \equiv \pm 1 \pmod{10}$, and $2.\mathfrak{A}_{5} \leq H/\cent{H}{W} \leq Z \circ 2.\mathfrak{A}_{5}$; 
			
			\item $53 \leq q=p \leq 2351$ and $|(H/\cent{H}{W})/\zent{H/\cent{H}{W}}| \leq 24$. 
		\end{enumerate}
	\end{enumerate}
\end{thm}

Part of the proof of Theorem \ref{orbit} (that is, following the proof of \cite[Theorem 6.1]{Mar16}) is Lemma \ref{seged}, which extends \cite[Lemma 4.1]{Mar16}.

\begin{lem}
	\label{seged}
	Let $(G,V)$ be a pair among the exceptions in Theorem \ref{orbit}, satisfying $n(G,V) < 2\sqrt{q-1}+1$. Then $|G| < |V|^{3/2}$.
\end{lem}

\subsection{Bounding $k(G)$}

We now also have to take $k(G)$ into account. 

\begin{thm}
	\label{main3}
	Let $V$ be an irreducible and faithful $FG$-module for some finite group $G$ and finite field $F$ of characteristic $p$ at least $53$. Suppose that $p$ does not divide $|G|$. Then we have at least one of the following. 
	\begin{enumerate}
		\item $n(G,V) \geq 2 \sqrt{p-1} + 1$.
		
		\item $k(G) \geq 2 \sqrt{p-1} + 1$.
		
		\item $|V| = |F| = p$.
		
		\item Case (2/a) of Theorem \ref{orbit} holds with $p=59$ and $1 < t \leq 15$, or $p=61$ and $t=1$, or $61 \leq p \leq 119$ and $2 \leq t \leq 5$.
		
		\item Case (2/b) of Theorem \ref{orbit} holds with $t \leq 4$. 
	\end{enumerate}
\end{thm}  

Theorem \ref{main3} is Theorem 7.1 in \cite{Mar16} with the following changes. One occurrence of $59$ is now $53$, in conclusions (1) and (2) the bound $2 \sqrt{p-1}$ is now $2 \sqrt{p-1} + 1$, there is a minor change in conclusion (3), and in conclusion (4) the $14$ became $15$ and the $4$ is now $5$. There are minor changes in the proof of Theorem \ref{main3} with respect to the proof of Theorem 7.1 of \cite{Mar16}. 

\subsection{Bounding $n(G,V)$ and $k(G)$}

Minor changes were made in the corresponding section of \cite{Mar16} to complete the proof of Theorem \ref{module}, using Theorem \ref{main3} and a theorem of Hering \cite[Chapter XII]{Hup-Bla82}.

\section{Known Results on principal blocks}\label{sec:princblockfolklore}
 
In the following we will denote by $B_0(G)$ the principal $p$-block of the group $G$.
Recall that 
$${1}_G \in \irr{B_0(G)}=\{ \chi \in \irr G \ | \ \sum_{{x \in G}\atop{(p, o(x))=1}}\chi(x)\neq 0 \}\, .$$
We collect useful facts about characters in principal $p$-blocks. 

 \begin{lem}\label{lem:princblockabove} Let $G$ be a finite group, let $N$ be a normal subgroup of $G$ and let $p$ be a prime. 
 \begin{enumerate}
 \item $\irr {B_0(G/N)}\sbs \irr{B_0(G)}$.
 \item If $N$ is a group of order coprime to $p$, then $\irr {B_0(G/N)}=\irr{B_0(G)}$.
 \item For any $\theta\in\irr{B_0(N)}$, there is some $\chi\in\irr{B_0(G})$ lying over $\theta$.
 \item If $B_0(G)$ is the only block covering $B_0(N)$, then for any  $\theta\in\irr{B_0(N)}$ every constituent of $\theta^G$ lies in $B_0(G)$. This happens, for instance, if $G/N$ is a $p$-group
 or if $P\cent G P \sbs N$ for some $P \in \syl p G$.
  \item For any  $\chi\in\irr{B_0(G)}$, every constituent of $\chi_N$ lies in $B_0(N)$.
 \end{enumerate}
 \end{lem}
 \begin{proof}
 Part (i) follows from the definition of block domination \cite[p. 198]{Nav98}, (ii) is \cite[Theorem 9.9(c)]{Nav98}, (iii) is \cite[Theorem 9.4]{Nav98}, (iv) is \cite[Theorem 9.6]{Nav98} and \cite[Lemma 1.3]{RSV21}, and (v) follows from \cite[Theorem 9.2 and Corollary 9.3]{Nav98} by noticing that $B_0(N)$ is always $G$-invariant.
 \end{proof}

\begin{lem}\label{lem:direct products}
Assume $G=H_1\times\dots\times H_t$. Then 
$$\Irr(B_0(G))=\{\theta_1\times\dots\times\theta_t\mid\theta_i\in\Irr(B_0(H_i))\}.$$
\end{lem}
\begin{proof}
This follows directly from the definition of principal block.
\end{proof}

 \begin{lem}[Alperin--Dade]\label{lem:alpdade}
 Let $N\lhd G$ and $p$ be a prime. Assume that $G/N$ is a group of order coprime to $p$ and that $G=N\cent{G}{P}$ for some $P \in \syl p N$. Then restriction yields a bijection between $\irr{B_0(G)}$ and $\irr{B_0(N)}$.
 \end{lem}
 \begin{proof}
 This is proved by Alperin \cite{alperin76} when $G/N$ is solvable, and by Dade \cite{dade77} in general. 
 \end{proof}

We also collect some results concerning the action of Galois automorphisms of $p$-power order on principal blocks. 
If $P$ is a $p$-group, then observe that a linear character $\lambda\in\Lin(P)$ is $\sigma$-invariant if and only if it has order $p$ or $1$. This happens if and only if $\Phi(P)\sbs\ker\lambda$. 

\begin{lem}\label{lem:normalsylow}
Let $G$ be a finite group of order divisible by $p$ and assume $G$ has a normal Sylow $p$-subgroup $P$. Then $\Irrg{B_0(G)}=\Irr({G/\oh{p'}G\Phi(P)})$.
\end{lem}
\begin{proof}
This is \cite[Lemma 2.2(a)]{Riz-Sch-Val20}.
\end{proof}

For a finite group $G$,
let $\tau$ be an automorphism of the cyclotomic field $\QQ_{|G|}$, assume that
$\tau$ fixes all $p'$-roots of unity in $\QQ_{|G|}$ and that $\tau$ has $p$-power order. If $p=2$, the $p$-power order condition is superfluous as the group ${\rm Gal}(\mathbb Q_{|G|}/\mathbb Q_{|G|_{2'}})$ is a 2-group. If $p$ is odd,
then $\tau \in \langle \sigma \rangle$, where $\sigma$ is as defined in the Introduction.

 \begin{lem}\label{lem:p'abelianquotsigmafixed}
Let $p$ be a prime and let $N\lhd G$. Assume that $G/N$ is a group of order coprime to $p$. 
If  $\chi \in \irr G$ and $\theta \in \irr N$ are such that $[\chi_N, \theta]\neq 0$, then
$\chi$ is $\tau$-fixed if, and only if, $\theta$ is 
$\tau$-fixed
\end{lem}
\begin{proof}
Let $G_\theta$ be the inertia group of $\theta$ and $\psi \in \irr{G_\theta | \theta}$ be the Clifford correspondent of $\chi$. 
If $\chi^\tau=\chi$, then $\tau|_{\mathbb Q(\psi)}\in {\rm Gal}(\Q(\psi)/\Q(\chi))$ has $p$-power order. By \cite[Lemma 2.1(ii)]{Nav-Tie21}, $\tau|_{\Q(\psi)}=1$ so $\psi^\tau=\psi$ and consequently
$\theta^\tau=\theta$. 
The reciprocal implication follows by direct application of \cite[Lemma 5.1]{Nav-Tie19}.
\end{proof}

\begin{cor}\label{cor:p'abelianquotsigmafixed}
Let $N \lhd G$ and $p$ be a prime. Assume that $G/N$ is a group of order coprime to $p$. 
\begin{enumerate}
\item If $\chi \in \irr{B_0(G)}$ is $\tau$-fixed, then every constituent $\theta \in \irr{B_0(N)}$ of $\chi_N$ is $\tau$-fixed.
\item If $\theta \in \irr{B_0(N)}$ is $\tau$-fixed then some constituent $\chi \in \irr{B_0(G)}$ of $\theta^G$ is $\tau$-fixed.
 \end{enumerate}
 Moreover, with the above notation, $\theta$ has degree coprime to $p$ if, and only if, $\chi$ has degree coprime to $p$.
\end{cor}
\begin{proof}
This follows directly from Lemmas \ref{lem:p'abelianquotsigmafixed} and \ref{lem:princblockabove}, together with the fact that $\chi(1)/\theta(1)$ divides $|G/N|$ for any $\chi$ and $\theta$ such that $[\chi_N, \theta]\neq 0$ by \cite[Theorem 5.12]{Nav18}.
\end{proof}

The following is \cite[Lemma 2.4]{Riz-Sch-Val20}. We give a proof here for completeness. 

\begin{lem}\label{lem:Cargument}
Let $p$ be a prime and let $H\leq G$. Assume that $p\nmid|G:H|$ and $\cent G Q\sbs H$ where $Q\in\Syl_p(H)$. Let $\theta\in\Irr_{p',\sigma}(B_0(H))$. Then there is some $\chi\in\Irr_{p',\sigma}(B_0(G))$ contained in $\theta^G$.
\end{lem}
\begin{proof}
Notice that $B_0(H)$ is admissible in the sense of \cite[p. 213]{Nav98}, so $B_0(H)^G$ is defined. By \cite[Theorem 6.7]{Nav98}, we have that $B_0(H)^G=B_0(G)$. Write $B_0=B_0(G)$  and
$$(\theta^G)_{B_0}=\sum_{\chi\in\Irr_{p'}(B_0)}a_\chi\chi+\sum_{\psi\in\Irr(B_0)\atop p\mid\psi(1)}a_\psi\psi\, ,$$
where $a_\chi=[\chi, \theta^G]$. 
As $p\nmid |G:H|\theta(1)=\theta^G(1)$, \cite[Corollary 6.4]{Nav98} implies that $$\sum_{\chi\in\Irr_{p'}(B_0)}a_\chi\chi(1)\not\equiv 0 \mod p.$$ Now let $\mathcal{O}_1,\dots,\mathcal{O}_t$ be the $\langle \sigma\rangle$-orbits on the set of characters in $\Irr_{p'}(B_0)$ lying over $\theta$. Since $\theta^G$ is $\sigma$-invariant, we have that if two characters $\chi,\eta\in\Irr_{p'}(B_0)$ lying over $\theta$ are $\langle \sigma\rangle$-conjugate then $a_\chi=a_\eta$. Thus if $\chi_i\in\mathcal{O}_i$ we have
$$\sum_{\chi\in\Irr_{p'}(B_0)}a_\chi\chi(1)=\sum_{i=1}^t|\mathcal{O}_i|a_{\chi_i}\chi_i(1)$$
is not divisible by $p$. Since $\langle \sigma\rangle$ is a $p$-group, this forces one of the $\mathcal{O}_i$ to have size $1$, as desired.
\end{proof}

Finally, we give an alternative proof of \cite[Lemma 2.5]{Riz-Sch-Val20}.

\begin{lem}\label{lem:G=PN}
Let $N\lhd G$ and $p$ be a prime. Assume $N\cong S^t$ is a minimal normal subgroup of $G$, where $S$ is a simple nonabelian group, and that $G/N$ is a $p$-group. Assume there is some $X$-invariant $\alpha\in\Irrg{B_0(S)}$ for $X/S\in\Syl_p(\Aut(S)/S)$. Then there is some $\chi\in\Irrg{B_0(G)}$ not containing $N$ in its kernel.
\end{lem}
\begin{proof}
Let $\mathcal{X}$ be the $\Aut(S)$-orbit of $\alpha$ and $\mathcal{Y}=\{\alpha_1\times\cdots\times\alpha_t\mid \alpha_i\in\mathcal{X}\}\sbs\Irrg{B_0(N)}$, so that $|\mathcal{Y}|=|\mathcal{X}|^t$ is not divisible by $p$. Since $N$ stabilizes every element of $\mathcal{Y}$ and $G/N$ is a $p$-group, there is some $G$-invariant $\lambda\in\mathcal{Y}$. Since $N$ is perfect, the determinantal order $o(\lambda)$ is $1$. By \cite[Corollaries 6.2 and 6.4]{Nav18} there is an extension $\hat\lambda\in\Irrg{G}$ of $\lambda$. Since $G/N$ is a $p$-group, $\hat\lambda\in\Irrg{B_0(G)}$ by Lemma \ref{lem:princblockabove}(iv).
\end{proof}

\section{Simple groups}\label{sec:simples}

The aim of this section is to prove parts (i) and (ii) from Theorem \ref{thm:simple groups} in the Introduction. Throughout, let $\mathcal{G}$ denote the absolute Galois group ${\rm Gal}(\QQ^{\rm ab}/\QQ)$ and given a prime $p$, let $\sigma\in\mathcal{G}$  be the Galois automorphism defined in the Introduction. 
 We begin with the following observation about sporadic and alternating groups.

\begin{lem}\label{lem:altspor}
Let $p$ be an odd prime, and let $S$ be an alternating group,  sporadic simple group, or the Tits group $\tw{2}\type{F}_4(2)'$. Assume $(p, S)\neq (3, J_3)$. Then every $\chi\in\irr{S}$ is almost $p$-rational.
\end{lem}
\begin{proof}
For the sporadic and Tits groups, we can see this directly from the Character Table Library in GAP \cite{GAP}. Now let $S=\mathfrak{A}_n$ be an alternating group and let $\chi$ be an irreducible character of $\aut{S}=\mathfrak{S}_n$, whose characters take only rational values. Then $\chi|_S$ contains at most two irreducible constituents, which must be acted on by $\sigma$. Since $\sigma$ has $p$-power order (when viewed as an element of $\mathrm{Gal}(\QQ_{|S|}/\QQ)$) and $p$ is odd, it follows that $\sigma$ acts trivially on these constituents, proving the claim.  
\end{proof}

Much of the work will be centered around groups of Lie type, so we begin by fixing some notation in that situation.

\subsection{Notation for groups of Lie type}
Before continuing, we set some notation for the remainder of this section, surrounding simple groups of Lie type.
 So for now, let $S$ be a simple group of Lie type. We will exclude the Tits group $\tw{2}\type{F}_4(2)'$ from this list, instead treating it separately with the sporadic groups. 
 
 We have $A:=\aut{S}=\wt{S}\rtimes D$, where $\wt{S}$ denotes the group of inner-diagonal automorphisms and $D$ is an appropriate group of graph-field automorphisms (see, e.g. \cite[Theorems 2.5.12 and 2.5.14]{GLS}).  Note that $\wt{S}/S$ is abelian, and that $D$ is abelian unless $S=\type{D}_4(q)$, in which case $D=\sym_3\times C$ for a cyclic group $C$. 

Let $\bG$ be a simple, simply connected linear algebraic group over $\bar\FF_{q_0}$ for some prime $q_0$, and let $F\colon \bG\rightarrow\bG$ be a Steinberg endomorphism such that $S=G/\zent{G}$ with $G=\bG^F$. Let $\mathbb{G}:=\bG_{\mathrm{ad}}$ be a simple algebraic group of adjoint type, of the same type as $\bG$.  Then $S\cong [\bbG^F, \bbG^F]$ and $\wt{S}\cong \bbG^F$ (see \cite[Prop. 24.21]{MT11} and discussion after). We will  make these identifications.

The following is an application of the work of Tiep--Zalesski \cite{TZ04} and will be useful in the case of defining characteristic.

\begin{lem}\label{lem:defcharfixed}
Let $S$ be as above, with $p=q_0\geq 3$.  Then every $\chi\in\irr{S}$ is $\sigma$-invariant.
\end{lem}
\begin{proof}
This follows from \cite[Thm. 1.3 and Prop. 10.12]{TZ04}.
\end{proof}

Now, let $(\bG^\ast, F)$ and $(\bbG^\ast, F)$ be dual to $(\bG, F)$ and $(\bbG, F)$, respectively. For each conjugacy class $(s)$ of semisimple elements (that is, $q_0'$-elements) in $(\bbG^\ast)^F$, there is a unique \emph{semisimple} character $\chi_s\in\irr{\wt{S}}$, whose degree is $|(\bbG^\ast)^F:\cent{\bbG^\ast}{s}^F|_{q_0'}$ (see \cite[Thm. 2.6.11]{GM20}.  These characters will play an important role in our proof.

The unipotent characters will also play an important role. Thanks to the work of Lusztig \cite{lus88}, the unipotent characters of $G$ are trivial on the center, so can be viewed as characters of $S$; furthermore, these characters extend to unipotent characters of $\wt{S}$. That is, the unipotent characters of $S$  can simultaneously be viewed as the deflation of unipotent characters of $G$ or as the irreducible restriction of unipotent characters of $\wt{S}$. In particular, the principal series unipotent characters of $G$ are those that are constituents of the induced character  $\Ind_B^G(1_B)$, where $B=\bg{B}^F$ is the set of fixed points of an $F$-stable Borel subgroup $\bg{B}$ of $\bG$.
By a slight abuse of terminology, if $\chi\in\irr{S}$ is the deflation of a principal series unipotent character of $G$, we will refer to $\chi$ itself as a principal series unipotent character.

\subsection{Part (i) of Theorem \ref{thm:simple groups}}

Although not completely necessary for our purposes here, we prove a stronger version of part (i) of Theorem \ref{thm:simple groups} in the case of groups of Lie type when $p\neq q_0$. We hope that this can be useful in future work. This statement will utilize the principal series unipotent characters.

\begin{lem}
\label{lem:extfield}
Let $S=G/\zent{G}$ be a simple group of Lie type, where $G=\bG^F$ with $\bG$  of simply connected type and $G\neq \tw{2}\type{F}_4(q^2)$.  Let $\chi\in\Irr(S)$ be  a principal series unipotent character. If $S\leq A\leq\aut{S}$ is such that $A$ is generated by inner-diagonal and field automorphisms, then there is an extension $\hat\chi$ of $\chi$ to $A$ that is $\mathcal{G}_\chi$-invariant.
\end{lem}
\begin{proof}
First, by \cite[Thms. 2.4 and 2.5]{Malle08}, every unipotent character extends to its stabilizer in $\aut{S}$, which includes the inner-diagonal and field automorphisms. If $F$ is a Frobenius morphism, then the statement follows directly from \cite[Prop. 2.6]{RSF22} and its generalization \cite[Prop. 8.7]{birtethesis}.  Otherwise, $G=\tw{2}\type{G}_2(q^2)$ or $\tw{2}\type{B}_2(q^2)$, and $\chi$ is either the trivial character or Steinberg character, so the result follows from \cite{schmid} in this case.
\end{proof}

\begin{lem}\label{lem:existsprincseriesunip}
Let $p$ be a prime and let $S=G/\zent{G}$ be a simple group of Lie type, where $G=\bG^F$ with $\bG$ of simply connected type defined in characteristic distinct from $p$. Then $\irrp{B_0(S)}$ contains a nontrivial rational-valued principal series unipotent character $\chi$ such that if  $S\leq A\leq\aut{S}$ is an almost simple group generated by inner-diagonal and field automorphisms, then there is a rational-valued extension of $\chi$ to $A$.
\end{lem}
\begin{proof}
If $p\in\{2,3\}$, then the Steinberg character $\mathrm{St}_S$ lies in $\irrp{B_0(S)}$. (Indeed, this can be seen using the arguments of \cite[Lem.~3.6]{RSV21} if $S$ is not a Ree group, or from \cite{Malle90, Ward} if $S$ is a Ree group.) Similarly, if $p\geq 5$ and $S$ is $\Sp_4(q)$ with $q$ even, $\PSL_2(q)$, $\PSL_3^\epsilon(q)$, $\type{D}_4(q)$, $\tw{3}\type{D}_4(q)$, or a Suzuki or Ree group, then, as  in the proof of \cite[Props. 4.4 and 4.5]{GRSS}, the Steinberg character $\mathrm{St}_S$ lies in $\irrp{B_0(S)}$.  Further, $\mathrm{St}_S$ is rational-valued and extends to a rational-valued character of $\aut{S}$, by \cite{schmid}.

From here, we may assume that $p\geq 5$ and $F$ is a Frobenius endomorphism, so that the extension property will follow from Lemma \ref{lem:extfield} once we have shown that $\irrp{B_0(S)}$ contains a rational-valued principal series unipotent character.

We next suppose that $S$ is of exceptional type $\type{G}_2(q)$, $\type{F}_4(q)$, $\type{E}_6(q)$, $\tw{2}\type{E}_6(q)$, $\type{E}_7(q)$, or $\type{E}_8(q)$. Here the principal series unipotent characters are rational except for (in the notation of \cite[Sec. 13]{carter}) the characters $\phi_{512, 11}, \phi_{512, 12}$ of $S=\type{E}_7(q)$ and $\phi_{4096,11}, \phi_{4096, 12}, \phi_{4096, }, \phi_{4096, }$ of $\type{E}_8(q)$, by \cite[Thm. 2.9]{bensoncurtis}
 and \cite{BensonCurtisCorrection}. Let $d$ be the order of $q$ modulo $p$. If $d$ is a regular number, then $B_0(S)$ again contains $\mathrm{St}_S$ (see, e.g. \cite[Lem. 3.6]{RSV21}, which follows directly from \cite[Thm. 6.6]{Malle07} and \cite[Theorem A]{Eng00}). Otherwise, we see using the decompositions of $d$-Harish-Chandra series in \cite[Table 2]{BMM} and the fact that all characters of such a series lie in the same block by \cite[Theorem A]{Eng00} to see that there is a rational-valued principal series unipotent character in $\irrp{B_0(S)}$.  

Finally, assume $G$ is one of the remaining cases of classical type. In these cases, every unipotent character is rational-valued (see \cite[Cor. 4.5.6]{GM20}). Here two unipotent members of $\irrp{B_0(S)}$ are illustrated in the proof of  \cite[Prop. 4.4]{GRSS}, given explicitly in Tables 1-4 of loc. cit. It suffices to know that at least one of these characters lies in the principal series. Here, \cite[Sec. 13.8]{carter} describes how to determine, from the combinatorial description  of the unipotent character, whether it lies in the principal series. From this, we see that, indeed, at least one of the two characters described in \cite[Tables 1-4]{GRSS} lies in the principal series in each case. 
\end{proof}

With this, we are prepared to prove part (i) of Theorem \ref{thm:simple groups}. 

\begin{proof}[Proof of Theorem \ref{thm:simple groups}(i)]
If $S$ is a sporadic, Tits, or alternating group, then every character is almost $p$-rational by Lemma \ref{lem:altspor}. Then since  $p\nmid |\Out(S)|$, it suffices in these cases to note that $B_0(S)$ contains a nontrivial $p'$-degree character.

Next, suppose that $S$ is a simple group of Lie type defined in characteristic $p$. We may find a simple, simply connected algebraic group $\bG$ defined over $\overline{\FF}_p$ such that $S=G/\zent{G}$, where $G=\bG^F$ for some Steinberg endomorphism $F$. Here again we have every character is almost $p$-rational by Lemma \ref{lem:defcharfixed}. In \cite[Props. 4.2 and 4.3]{GRSS} and building off of \cite[Prop. 4.3]{GRS}, a nontrivial member of $\irrp{B_0(S)}$ is exhibited that extends to $X$. 

Finally, assume that $S$ is a simple group of Lie type defined in characteristic distinct from $p$. Then since $p\geq 5$ and the graph automorphisms of $S$ have order at most $3$, we have $X$ such that $X/S\in\Syl_p(A/S)$ can be chosen to be generated by inner-diagonal and field automorphisms, so the statement follows from Lemma \ref{lem:existsprincseriesunip}.
\end{proof}

\subsection{Part (ii) of Theorem \ref{thm:simple groups}}
We  now turn our attention toward part (ii) of Theorem \ref{thm:simple groups}. In many cases, we will be able to use the following simplified condition:

\begin{cor}\label{cor:Biisimplified}\
Let $S\leq T$ be as in Theorem \ref{thm:simple groups}(ii). Then Theorem \ref{thm:simple groups}(ii) holds for $T$ if there exist at least $2\sqrt{p-1}+1$ characters in $\Irrg{B_0(S)}$ that are not $T$-conjugate.
\end{cor}
\begin{proof}
This follows from Corollary \ref{cor:p'abelianquotsigmafixed}, and Clifford theory, as non-$T$-conjugate members of $\irr{S}$ cannot lie under the same character in $\irr{T}$.
\end{proof}

Lemma \ref{lem:defcharfixed} will continue to be useful in defining characteristic. In the case of non-defining characteristic, we will use unipotent characters and certain semisimple characters to obtain the required bounds.

\begin{lem}\label{lem:semisimplecond}
Let $p$ be a prime and assume that $p\neq q_0$. Let $\wt\chi\in\irr{\wt{S}}$ be semisimple corresponding to a semisimple class containing $s\in\bbG^\ast$ such  $|s|=p$. Then $\wt{\chi}\in\Irr_{\sigma}(B_0(\wt{S}))$.

Further, if $p\geq 5$ and either $S\neq \PSL_n^\epsilon(q)$ or  $p\nmid (q-\epsilon)$, then $\wt\chi$ restricts irreducibly to $S$ and is not $D$-conjugate to $\wt\chi\beta$ for any nontrivial $\beta\in\irr{\wt{S}/S}$. 
\end{lem}
\begin{proof}
Since $|s|=p$, we have that the semisimple character $\chi_s\in\irr{\wt{S}}$ corresponding to $s$ is fixed under $\sigma$, by e.g. \cite[Lem. 3.4]{SFT18}. Further, $\chi_s$ lies in $B_0(\wt{S})$ by \cite[Cor. 3.4]{hiss90}, so $\chi_s\in\Irr_\sigma(B_0(\wt{S}))$.

Now, assume $p\geq 5$ and $S\neq \PSL_n^\epsilon(q)$ or  $p\nmid (q-\epsilon)$ and suppose that $\chi_s^\varphi=\chi_s\beta$ for some $\varphi\in D$ (with the possibility $\varphi=1$) and $1\neq \beta\in\irr{\wt{S}/S}$. Then $\chi_s^\varphi=\chi_{sz}$ for some $1\neq z\in \zent{(\bbG^\ast)^F}$ using \cite[Prop. 2.5.20]{GM20}. Our conditions on $p$ ensure that $p\nmid |z|$ for any such $z$, and hence $|s|\neq |sz|$. But we have $s^{\varphi^\ast}$ is conjugate to $sz$ for some automorphism $\varphi^\ast$ dual to $\varphi$ by \cite[Prop. 7.2]{tay18}, contradicting that $|s|\neq |sz|$. We therefore obtain that no such $\varphi, \beta$ exist, so $\chi_s$ is not $D$-conjugate to $\chi_s\beta$ with $1\neq \beta$. Further, $\chi_s$ is irreducible on restriction to $S$, since the number of constituents in the restriction is the number of $\beta$ such that $\chi_s=\chi_s\beta$ using Clifford theory, since $\wt{S}/S$ is abelian and is multiplicity-free by \cite[Thm. 1.7.15]{GM20}. 
\end{proof}

We will also use the fact that unipotent characters are almost $p$-rational:
\begin{lem}\label{lem:unipcond}
Let $p\geq 3$ be a prime distinct from $q_0$. Then any unipotent character of $S$ or $\wt{S}$ is invariant under $\sigma$.
\end{lem}
\begin{proof}
For a positive integer $m$, let $\zeta_m$ denote a primitive $m$th root of unity. First, note that $\sqrt{q}$ lies in $\QQ(\zeta_{q_0}, i)$, which is fixed under $\sigma$ since $p\neq q_0$ and $p\geq 3$. Further, $\sigma$ fixes $\zeta_r$ for any prime $r$. Then from \cite[Table 1 and Prop. 5.6]{Geck03}, we see that the statement holds.
\end{proof}

\medskip

With this, we are ready to prove Theorem \ref{thm:simple groups}(ii).

\begin{proof}[Proof of Theorem \ref{thm:simple groups}(ii)]
If $S$ is sporadic, alternating $\mathfrak{A}_n$ with $n\leq 8$, Tits, or a group of Lie type with exceptional Schur multiplier (see \cite[Tab. 6.1.3]{GLS}), {this can be checked in GAP \cite{GAP}.}  
If $S=\mathfrak{A}_n$ with $n\geq 9$, then the result follows from \cite[Section 3]{Hun-Sch23} and the ``if" direction of Brauer's Height Zero conjecture \cite{KM13} if a Sylow $p$-subgroup of $S$ is abelian, and by \cite[Proposition 3.1(iii)]{Hun-Sch-Val23} otherwise, taking into account Lemma \ref{lem:altspor}.

\medskip

We now let $S$ be a simple group of Lie type, defined in characteristic $q_0$. Let $P\in\Syl_p(S)$ and assume $P$ is non-cyclic. Note that since $\tw{2}\type{B}_2(q^2)$ and $\tw{2}\type{G}_2(q^2)$ have cyclic Sylow $p$-subgroups for $p\geq 5$ (see, e.g. \cite[Theorem 4.10.2]{GLS}), we may assume throughout that $S$ is not one of these. 

In most cases, we aim to exhibit at least $2\sqrt{p-1}+1$ characters in $\Irrg{B_0({S})}$  that are not $\aut{S}$-conjugate and apply Corollary \ref{cor:Biisimplified}. Often, we will use the results of \cite{Hun-Sch23, Hun-Sch-Val23}, together with Lemmas \ref{lem:defcharfixed}, \ref{lem:unipcond}, and \ref{lem:semisimplecond}.
 It will be useful to note throughout that $p\geq 2\sqrt{p-1}+1$ for $p\geq 5$ and that $x+\frac{p-1}{x}+1\geq 2\sqrt{p-1}+1$ for $p\geq 5$ and $x\geq 1$.

\medskip

We first assume that $q_0\neq p$. 

\begin{proof}[Non-Defining Characteristic with $P$ Abelian]

Let $P\in\Syl_p(S)$ be abelian. Note that by the ``if" direction of  BHZ \cite{KM13}, we have $\irrp{B_0(S)}=\irr{B_0(S)}$, and similar for $T$.

With this, the proof in this case follows almost directly from the arguments in the proof of \cite[Theorem 1.2]{Hun-Sch23}, with some minor modification. In loc. cit., the strategy is to exhibit enough unipotent characters and semisimple characters $\chi_s$ with $s\in G^\ast$ a $p$-element that lie in $B_0(S)$ and are not $\aut{S}$-conjugate. However, we see that by replacing arbitrary $p$-elements $s$ with elements only of order $p$, the bounds obtained there are still sufficient, and now the characters lie in $\Irrg{B_0(S)}$ thanks to Lemmas \ref{lem:unipcond} and \ref{lem:semisimplecond}. 
We summarize the details. 
Throughout, we will let $\mathscr{U}$ denote the set of unipotent characters in $B_0(S)$ and $\mathscr{S}$ the set of characters in $B_0(S)$ that are the irreducible restrictions of semisimple characters $\chi_s\in\irr{\wt{S}}$ with $s$ of order $p$. (That is, $\mathscr{U}\cup\mathscr{S}$ is the set of  characters for which we will   be able to apply Lemmas \ref{lem:unipcond} and \ref{lem:semisimplecond}.) 

Let $e$ denote the order of $q$ modulo $p$ and  let $\Phi_e$ be the $e$th cyclotomic polynomial in $q$ if $S\neq \tw{2}\type{F}_4(q^2)$. If $S=\tw{2}\type{F}_4(q^2)$, then let $\Phi_e$ instead be the polynomial $\Phi^{(p)}$ over $\QQ(\sqrt{2})$ defined in \cite[Sec. 8.1]{Malle07}. 
Let $p^a$ be the largest power of $p$ dividing $\Phi_e$, and let $\Phi_e^k$ be the order of a Sylow $\Phi_e$-torus of $(\bbG, F)$. Then a Sylow $p$-subgroup $\wt{P}$ of $\wt{S}$ and of $(\bbG^\ast)^F$ is isomorphic to $\wt P\cong C_{p^a}^k$ by \cite[Lem. 2.1 and Prop. 2.2]{Malle14}. (Note here that the situation that $P$ is abelian but $\wt{P}$ is not is discussed after \cite[Prop. 2.2]{Malle14}, but these cases do not occur here since $p\geq 5$.)
The analogue of \cite[Lemma 5.3]{Hun-Sch23} holds for elements $t\in \bbG^\ast$ of order $p$, yielding that the corresponding semisimple character $\chi_t$ of $\wt{S}$ has an orbit of size at most $p-1$ under a field automorphism $\alpha$. 
With this, the considerations in \cite[5D]{Hun-Sch23} yield the same bound as \cite[(5.3)]{Hun-Sch23}. Namely, the number of $\aut{S}$-orbits on $\Irrg{B_0(S)}$ from semisimple characters  is at least $\frac{p^k-1}{dg(p-1)|W_e|}$, where $W_e$ is the so-called relative Weyl group for a minimal $e$-split Levi subgroup of $(\bbG, F)$, $d=|\wt{S}/S|$, and $g$ is the size of the group of graph automorphisms in $D$. Note that $k\geq 2$, since by assumption the Sylow $p$-subgroups of $S$ are non-cyclic.

From here, the proof in \cite[Theorem 8.1]{Hun-Sch23} gives Theorem \ref{thm:simple groups}(ii)(b) when $S$ is an exceptional group of Lie type. To obtain part (a), we again follow the proof of \cite[Theorem 8.1]{Hun-Sch23}, where the methods there in fact  yield our required $|\Irrg{B_0(T)}|\geq  2\sqrt{p-1}+1$, although only $2\sqrt{p-1}$ was required in \cite{Hun-Sch23}. 

Now suppose that $S$ is of classical type. First let $S=\PSL_n^\epsilon(q)$ with $n\geq 3$. (Note that  $\PSL_2(q)$ has cyclic Sylow $p$-subgroups.) Let $e'$ be the order of $\epsilon q$ modulo $p$, and let $w=\lfloor \frac{n}{e'}\rfloor$. Since $n\neq 2$, note that we have $(e', w)\neq (1,2)$. Here the proof of \cite[Prop. 6.1]{Hun-Sch23} yields the result. There we see that the number of $\aut{S}$-orbits in $\mathscr{U}$ is at least $2e'+1$ and that the number of $\aut{S}$-orbits in $\mathscr{S}$ is at least $\frac{p-1}{2e'}$. Then we have at least $2e'+\frac{p-1}{2e'}+1\geq 2\sqrt{p-1}+1$ $\aut{S}$-orbits in $\Irrg{B_0(S)}$, using Lemmas \ref{lem:unipcond} and \ref{lem:semisimplecond}.

Next, we let $S=\type{C}_n(q)$ with $n\geq 2$, $\type{B}_n(q)$ with $n\geq 3$, $\type{D}_n(q)$ with $n\geq 4$, or $\tw{2}\type{D}_n(q)$ with $n\geq 4$. Here let $e'$ denote the order of $q^2$ modulo $p$ and again let $w=\lfloor \frac{n}{e'}\rfloor\geq 2$. Then  the proof of \cite[Prop. 7.1]{Hun-Sch23} shows that the number of $\aut{S}$-orbits in $\mathscr{U}$ is at least $4e'$, and is strictly greater unless $S\in\{\type{D}_n(q), \tw{2}\type{D}_n(q)\}$ with $(e', w)=(2,2)$ or $S=\type{C}_2(2^f)$. 

If $S=\type{C}_n(q)$ or $\type{B}_n(q)$, the proof in loc. cit. then tells us that there are at least $4e'+\frac{p-1}{4e'}+1\geq 2\sqrt{p-1}+1$ orbits under $\aut{S}$ of characters in $\mathscr{U}\cup\mathscr{S}$, unless   $S=\type{C}_2(2^f)$. In the latter case, we instead have at least $4+\frac{p-1}{8}\geq 2(p-1)^{1/4}$ such orbits, giving (b). However, the proof there also shows that the number of characters in $B_0(T)$ lying above members of $\mathscr{U}\cup\mathscr{S}$ is at least $3b+\frac{(p-1)^2}{8b}$, where $b=|T/S\cent{T}{P}|$. Since this number is at least $p$ for $p\geq 5$ and $b\geq 1$, we are again done in this case.

Now let $S=\type{D}_n(q)$ with $n\geq 5$ or $S=\tw{2}\type{D}_n(q)$ with $n\geq 4$. In this case, the proof of \cite[Prop. 7.1]{Hun-Sch23} gives at least $4e'+\frac{p-1}{4e'}+1\geq 2\sqrt{p-1}+1$ orbits under $\aut{S}$ of characters in $\mathscr{U}\cup\mathscr{S}$ unless $e'=2$. Suppose we are then in the case $e'=2$. Let $\Gamma$ be the group $\wt{S}\leq \Gamma\leq \aut{S}$ generated by inner, diagonal, and graph automorphisms, and consider the group $X:=(\Gamma\cap T)\cent{T}{P}$.
Then the Alperin--Dade correspondence Lemma \ref{lem:alpdade} gives a bijection, via restriction, between $\irr{B_0(X)}$ and $\irr{B_0(\Gamma\cap T)}$. Further, $B_0(T)$ is the unique block covering $B_0(X)$ by Lemma \ref{lem:princblockabove}(iv). Since there are at least $8$ $\aut{S}$-orbits of characters in $\mathscr{U}$, we have at least $8$ characters in $\Irrg{B_0(X)}$ lying over members of $\mathscr{U}$, using Lemma \ref{lem:unipcond} and Corollary \ref{cor:p'abelianquotsigmafixed}. 
Letting $b:=|T/X|$, this gives at least $8b$ characters in $\Irrg{B_0(T)}$ lying above these, since the unipotent characters are fixed by field automorphisms and extend to their inertia groups in $\aut{S}$ by \cite[Theorems 2.4 and 2.5]{Malle08}. Further, there are at least $\left(\frac{p-1}{4}\right)^2$ characters in $\mathscr{S}$ using the arguments of \cite[Theorem 7.1]{Hun-Sch23}. This gives at least $\frac{1}{8}\left(\frac{p-1}{4}\right)^2$ characters in $\Irrg{B_0(X)}$, and $\frac{1}{8b}\left(\frac{p-1}{4}\right)^2$ in $\Irrg{B_0(T)}$, lying above these. Hence, we obtain that $|\Irrg{B_0(T)}|\geq 8b+\frac{(p-1)^2}{128b}$, which is larger than $  2\sqrt{p-1}+1$ for $p\geq 5$ and $b\geq1$. 

Finally, consider $S=\type{D}_4(q)$. Here the methods of the penultimate paragraph of the proof of \cite[Prop. 7.1]{Hun-Sch23} yield at least $4+\frac{p-1}{24}\geq 2(p-1)^{1/4}$ orbits under $\aut{S}$ in $\mathscr{U}\cup\mathscr{S}$, giving (b). For (a), the final paragraph of loc. cit. gives at least $4b+\frac{(p-1)^2}{96b}$  characters in  $B_0(T)$ lying above members of $\mathscr{U}\cup\mathscr{S}$, where $b=|T/X|$ with $X=(\Gamma\cap T)\cent{T}{P}$ and $\Gamma=\wt{S}\rtimes \mathfrak{S}_3$ is the group generated by $\wt{S}$ and the group of graph automorphisms. 
This number is at least $2\sqrt{p-1}+1$ for $p\geq 5$ and $b\geq 2$ or for $p>23$ and $b\geq 1$. At this point, we remark that this bound can be improved to $4b+\frac{(p-1)^2}{48b}+\frac{(p-1)}{12b}$, since the original bound in \cite{Hun-Sch23} only considered semisimple elements corresponding to elements $(\lambda_1, \lambda_2, \ldots, \lambda_k)$ of $C_p^k$ with at most $k-2$ values of $i$ such that $\lambda_i=1$, and since the bound $\frac{(p-1)^2}{16}$ for the number of orbits of semisimple characters was from restriction from the group $\operatorname{GO}_{8}(q)$, which therefore had already taken into account the action of an order-two graph automorphism.  This number is more than $\lceil 2\sqrt{p-1}\rceil$ unless $p=11$ with $b=1$, in which case the order $3$-graph automorphism must act trivially on the semisimple elements of order $11$ corresponding to elements as above with $\lambda_i=1$ for at least $k-2$ values of $i$. Then since $4+10^2/16+10/4>\lceil 2\sqrt{10}\rceil$, we are again done in this case. 
\phantom\qedhere
\end{proof}

\begin{proof}[Non-Defining Characteristic with $P$ Nonabelian]

Let $P$ be nonabelian. Then if $S$ is of exceptional type, 
this means that $p\in\{5,7\}$ and $\bG$ is type $\type{E}_6$, $\type{E}_7$, or $\type{E}_8$, and hence it suffices to know that there are at least 6 unipotent characters in $\irrp{B_0(\wt{S})}$ that are not $\Aut(S)$-conjugate, applying Lemma \ref{lem:unipcond} and Corollary \ref{cor:Biisimplified}. The existence of these characters (recalling that $P\in\Syl_p(S)$ is noncyclic and $p\geq 5$) is proven in \cite[Lem. 3.7]{RSV21}.

We now assume that $S$ is of classical type.  In this case, parts (IV) and (VI) of the proof of \cite[Prop. 3.2]{Hun-Sch-Val23} show that there are at least $p$  unipotent characters in $\irrp{B_0(S)}$ that are not $\Aut(S)$-conjugate, and we are again done by applying Lemma \ref{lem:unipcond}.

\phantom\qedhere
\end{proof}

Finally, suppose that $q_0=p$.

\begin{proof}[Defining Characteristic]
 Using Lemma \ref{lem:defcharfixed} and Corollary \ref{cor:Biisimplified}, we are almost done using \cite[Prop. 3.2]{Hun-Sch-Val23}. However, note that we must exibit possibly one additional character to what was needed in loc. cit. Using part (II) of that proof, we have $|\Irrg{B_0(S)}|\geq q^r/d$, where $r$ is the semisimple rank of $G$ and $d=|\wt{S}/S|$ is the size of the diagonal automorphism group. Then the number of $\aut{S}$-orbits on $\Irrg{B_0(S)}$ is at least $\frac{q^r}{d\cdot|\Out(S)|}=\frac{q^r}{d^2\cdot f\cdot g}\geq \frac{q^{r/2}}{d^2\cdot g}$, where $q=p^f$ and $g$ is the size of the group of graph automorphisms, noting that $q^r/f=p^{rf}/f\geq \sqrt{q^r}$. Now, from the knowledge of $d$ and $g$ from, e.g., \cite[Thm. 2.5.12]{GLS}, we have $\frac{q^r}{d^2\cdot f\cdot g}\geq p$, except possibly if $S=\PSL_4^\epsilon(5)$; $S=\PSL_2(q)$ with $q\in\{p, 25, 49\}$; or $S=\PSL_3^\epsilon(p)$. Note that $\PSL_2(p)$ has a cyclic Sylow $p$-subgroup. In the remaining outlying cases, we have $\frac{q^r}{d^2\cdot f\cdot g}\geq 2\sqrt{p-1}+1$, except possibly for $\PSL_4^\epsilon(5)$, $\PSL_3^\epsilon(5)$, $\PSL_3^\epsilon(7)$, $\PSL_3^\epsilon(11)$, and $\PSL_2(25)$. In these cases, we see directly from their character tables in GAP that there are at least $1+2\sqrt{p-1}$ orbits under $\aut{S}$ in $\Irrg{B_0(S)}$. 
\phantom\qedhere
\end{proof}

This completes the proof of Theorem \ref{thm:simple groups}(ii).
\end{proof}

\section{Cyclic Sylow $p$-subgroups}\label{sec:cyclicSylow}

The aim of this section is to prove part (iii) from Theorem B of the Introduction.  
We make use of the fact that simple nonabelian groups with cyclic Sylow 
$p$-subgroups satisfy the inductive Alperin--McKay conditions (as defined in \cite{Spa13}) 
for odd primes \cite{Kos-Spa16},
together with deep properties of block character triple isomorphisms from \cite{Nav-Spa14}.
We use the notation $\sim_b$ from \cite[Definition 3.6]{Nav-Spa14}.

Assume $S$ has a cyclic Sylow $p$-subgroup $P$ with $p$ an odd prime. 
Then following \cite[Section 5]{Kos-Spa16} the characters in $B_0=B_0(S)$ can be written as a disjoint
union
$$\Irr_{\mathrm{ex}}(B_0)\cup\Irr_{\mathrm{nex}}(B_0)$$
where $\Irr_{\mathrm{nex}}(B_0)$ is exactly the set of $p$-rational characters of $B_0$ by \cite[Lemma 5.5]{Kos-Spa16}. 
In particular, $1_S \in \Irr_{\mathrm{nex}}(B_0)$. If $\Lambda$ is an $\norm G P$-transversal in $\Irr(P)\setminus\{1_P\}$ 
then for $\lambda \in \Lambda$, we define
$$\eta_\lambda=\sum_{g\in [\norm G P/\cent G P]}\lambda^g,$$
where $[\norm G P/\cent G P]$ is a transversal of $\cent G P$ in $\norm G P$. 
Notice that $\eta_\lambda$ does not depend on the representative $\lambda\in \Lambda$.

Fixing some $\chi \in \Irr_{\mathrm{nex}}(B_0)$, we have a parametrization
$\Irr_{\mathrm{ex}}(B_0)=\{ \chi_{\lambda}\mid \lambda\in\Lambda\}$
where, by \cite[Lemma 5.6]{Kos-Spa16} $\chi_\lambda$ is the unique non-$p$-rational constituent of the generalized character
$$\chi*\eta_{\lambda},$$ defined in the main theorem of \cite{Bro-Pui80} (\cite[p. 114]{Nav98}). 
For principal blocks, the $*$ construction can be simplified. If $x$ is a $p$-element, we denote by $S(x)=\{u\in G\mid u_p \text{ is $G$-conjugate to } x\}$ (this is known as the $p$-section of $x$ \cite[p. 105]{Nav98}).

\begin{lem}\label{lem:star formula} Let $G$ be a finite group, $p$ a prime, $\chi\in\Irr(B_0(G))$ and $P \in \syl p G$. Let $\mathbb{S}$ 
be a $G$-transversal on the set of $p$-elements of $G$ contained in $P$, then
$$\chi*\eta_{\lambda}=\sum_{x\in \mathbb{S}}\eta_{\lambda}(x)\mathbf{1}_{S(x)}\chi$$
where $\mathbf{1}_{S(x)}$ is the characteristic function of the $p$-section $S(x)$ of $x$. 
\end{lem}
\begin{proof}
Let $g \in G$, then $g \in S(x_0)$ for a unique $x_0 \in \mathbb S$. In particular, $g^u=x_0y$ for some $u$ in $G$, 
for a unique $x_0 \in \mathbb S$ and some element $y \in \cent G {x_0}$ of order coprime to $p$. 
Then
$$\chi*\eta(g)=\chi*\eta(xy)=\sum _{x \in {\mathbb S}}\eta (x) \chi^{(x, b_0(x))}$$
where $b_0(x)=B_0(\cent G x)$.

Now as in \cite[p. 114]{Nav98}, $\chi^{(x, b_0(x))}(g)=0$ if $g$ is not in $S(x)$ and if $g \in S(x)$ and $g^u =xy$ with   $y \in \cent G x$ 
of order coprime to $p$, we have
$$\chi^{(x, b_0(x))}(g)=\sum _{\varphi \in {\rm IBr}(b_0(x))}d^x_{\chi\varphi}\varphi(y)$$
where $d^x_{\chi\varphi}$ are the generalized decomposition numbers \cite[p. 100]{Nav98}.
By Brauer's Third Main theorem \cite[Theorem 6.7]{Nav98} and \cite[Corollary 5.8]{Nav98} we have
$$\sum _{\varphi \in {\rm IBr}(b_0(x))}d^x_{\chi \varphi}\varphi(y)=\chi(xy)=\chi(g)$$
and we are done.
\end{proof}

We can now easily check that the bijection constructed in  \cite[Section 6]{Kos-Spa16} is equivariant with respect to 
the action of $\langle \sigma \rangle$ on characters.

\begin{lem}\label{lem:cyclic bijection galois}
Let $p$ be an odd prime. Let $S$ be a finite group of order divisible by $p$. Let $P\in \syl p S$. Suppose that $P$ is cyclic. 
Then the bijection 
$$\Omega:\Irr(B_0(S))\rightarrow\Irr(B_0(\norm S P))$$
from \cite[Section 6]{Kos-Spa16} is $(\langle\sigma\rangle\times \norm {\aut S} P)$-equivariant. 
Moreover, every character in $\Irr(B_0(S))$ is almost $p$-rational if, and only if, $|P|=p$.
\end{lem}
\begin{proof}
Write $H=\norm S P$, let $\chi\in\Irr_{\mathrm{nex}}(B_0(S))$ and $\Omega(\chi)=\chi'\in\Irr_{\mathrm{nex}}(B_0(H))$.
In the bijection $\Omega$ from \cite[Section 6]{Kos-Spa16}, $\Irr_{\mathrm{nex}}(B_0(S))$ maps onto 
$\Irr_{\mathrm{nex}}(B_0(H))$ and the restriction $\Omega:\Irr_{\mathrm{ex}}(B_0(S))\rightarrow \Irr_{\mathrm{ex}}(B_0(H))$
 is defined by $\Omega(\chi_{\lambda})=\chi'_{\lambda}$ with the labeling from the beginning of this section. 
 Since the $p$-rational characters of $\Irr(B_0(S))$ are exactly the characters in $\Irr_{\mathrm{nex}}(B_0(S))$
  it follows that to check $\langle\sigma\rangle$-equivariance it suffices to check it for the exceptional characters.

It is straightforward to check that $\eta_{\lambda}^\sigma=\eta_{\lambda^\sigma}$. 
Using Lemma \ref{lem:star formula} and the fact that $\chi\in\Irr_{\mathrm{nex}}(B_0(S))$
 is $p$-rational we obtain that $\chi *(\eta_\lambda)^\sigma=(\chi*\eta_{\lambda})^\sigma$.
 Similarly, $\chi' *(\eta_\lambda)^\sigma=(\chi'*\eta_{\lambda})^\sigma$. 
 Now by \cite[Lemma 5.6]{Kos-Spa16} $\chi_{\lambda}$ is the unique non-$p$-rational
  constituent of $\chi*\eta_{\lambda}$, so $\Omega(\chi^\sigma)=\Omega(\chi)^\sigma$,
   and the $\langle\sigma\rangle$-equivariance of $\Omega$ is proven.
   Furthermore, $\Omega$ is $\norm{\Aut(S)}{P}$-equivariant by \cite[Proposition 6.1]{Kos-Spa16}.
The second part of the statement follows from the first one. 
\end{proof}

\begin{pro}\label{pro:triple isomorphism}
Let $N \triangleleft G$ and let $p$ be a prime. Suppose that $N=S_1\times\dots\times S_t$ where $S_i$ are nonabelian simple groups with cyclic nontrivial Sylow $p$-subgroups. Let $P\in\Syl_p(N)$. Then there is a $\langle \sigma \rangle \times\norm G P$-equivariant bijection
$$\Omega:\irr{B_0(N)}\rightarrow\irr{B_0(\norm N P)}$$
such that if $\theta\in\irr{B_0(N)}$ then
$$(G_{\theta},N,\theta)\sim_b(\norm G P_\theta,\norm N P,\Omega(\theta)).$$
\end{pro}
\begin{proof} Notice that under our assumptions $p$ must be odd (using \cite[Corollary 5.14]{Isa08}).
 Let $A_i=\Aut(S_i)$. Write $P=Q_1\times\dots\times Q_t$ where $Q_i\in\Syl_p(S_i)$. By Lemma \ref{lem:cyclic bijection galois} there is a $\langle\sigma\rangle\times\norm{A_i}{Q_i}$-equivariant bijection
$$\Omega_{S_i}:\Irr(B_0(S_i))\rightarrow\Irr(B_0(\norm {S_i}{Q_i}))\, .$$
By \cite[Theorem 7.6]{Kos-Spa16} $B_0(S_i)$ satisfies the inductive Alperin--McKay conditions and the character bijection is exactly $\Omega_{S_i}$.

Using Lemma \ref{lem:direct products} the construction of \cite[Theorem 7.9]{Spa13} and \cite[Theorem 6.3]{Nav-Spa14} we have that the map
$$\Omega:\irr{B_0(N)}\rightarrow\irr{B_0(\norm N P)}$$
defined by $$\Omega(\eta_1\times\dots\times \eta_t)=\Omega_{S_1}(\eta_1)\times\dots\times\Omega_{S_t}(\eta_t)$$
is $\langle \sigma \rangle\times \norm{G}{P}$-equivariant
and whenever $\theta\in\irr{B_0(N)}$ we have that
$$(G_{\theta},N,\theta)\sim_b(\norm G P_\theta,\norm N P,\Omega(\theta)) \, ,$$
as wanted.
\end{proof}

We now show that the principal $p$-block of groups $G$ with a semisimple normal subgroup $N$ 
of index coprime to $p$ and whose simple factors have cyclic Sylow $p$-subgroups satisfies \cite[Conjecture D]{Isa-Nav02}.  If $S$ is a finite group with cyclic Sylow $p$-subgroups, then $\irr{B_0(S)}$ consists only of characters of $p'$-degree by work of Dade \cite{Dad66}. It follows from Lemma \ref{lem:direct products} that if $N=S\times\dots\times S$ then $\irr{B_0(N)}$ also contains only characters of $p'$-degree. By Lemma \ref{lem:princblockabove}(iv) and Corollary \ref{cor:p'abelianquotsigmafixed} we conclude that, for groups $X$ as in Theorem \ref{thm:simple groups}(iii), $\irr{B_0(X)}$ contains only characters of $p'$-degree.

As usual, if $N\normal G$ and $\theta\in\Irr(N)$, we denote by $\Irr(G|\theta)$ the set of characters of $G$ that lie over $\theta$. We denote $\Irr(B_0(G)|\theta)=\Irr(B_0(G))\cap\Irr(G|\theta)$. Notice that $\Irr(B_0(G)|\theta)$ is not empty if and only if $\theta\in\Irr(B_0(N))$ by Lemma \ref{lem:princblockabove}(iii) and (v).

\begin{pro}\label{thm:simple groups ConjectureD}
Let $N\triangleleft G$ and let $p$ be a prime. Suppose that $N$ is semisimple, its simple factors have nontrivial cyclic Sylow $p$-subgroups, and that $p\nmid|G:N|$. Let $P\in \syl p G$. Let $\tau$ be an automorphism of the cyclotomic field $\QQ_{|G|}$ and assume that
$\tau$ has $p$-power order and that it fixes all $p'$-roots of unity in $\QQ_{|G|}$. Then $\tau$ fixes
equal numbers of (height zero) characters in $\irr{B_0(G)}$ and $\irr{B_0(\norm G P)}$.
\end{pro}
\begin{proof} Notice that under our assumptions $p$ must be odd. Then $\tau \in \langle \sigma \rangle$. Write $M=\norm N P$ and $H=\norm G P$, so that $G=NH$ and $M=N\cap H$.
By Proposition \ref{pro:triple isomorphism} there is an $H \times \langle \sigma \rangle$-equivariant bijection
 $$\Omega\colon \Irr(B_0(N))\rightarrow\Irr(B_0(M))$$
 such that $$(G_\theta, N, \theta)\sim_b(H_\theta, M, \Omega(\theta))$$
  for every $\theta\in \Irr(B_0(N))$. By the definition of $\sim_b$, \cite[Definition 3.6]{Nav98} and Brauer's Third Main theorem \cite[Theorem 6.7]{Nav98}, this implies that
  $$|\Irr(B_0(G_\theta)|\theta)|=|\Irr(B_0(H_\theta)|\Omega(\theta))|$$
 for every $\theta\in \Irr(B_0(N))$. 
 Let $\Delta$ be a complete set of representatives of the $H$-orbits on $\Irr_\tau(B_0(N))$. Since $\tau \in \langle \sigma \rangle$, we have that $\Omega(\Delta)$ is a complete set of representatives of the $H$-orbits on $\Irr_\tau(B_0(M))$. Then, by Corollary \ref{cor:p'abelianquotsigmafixed}, we have the partitions
 $$\Irr_{\tau}(B_0(G))=\coprod_{\eta\in\Delta}\Irr(B_0(G)|\eta)$$
 and 
 $$\Irr_{\tau}(B_0(H))=\coprod_{\eta\in\Delta}\Irr(B_0(H)|\Omega(\eta))\, .$$
 By the Fong--Reynolds correspondence \cite[Theorem 9.14]{Nav98} we obtain
 $$|\Irr(B_0(G)|\eta)|=|\Irr(B_0(G_\eta)|\eta)|=|\Irr(B_0(H_\eta)|\Omega(\eta))|=|\Irr(B_0(H)|\Omega(\eta))|\, $$
 as desired.
\end{proof}

The statement of Theorem B(iii) is now a corollary.

\begin{cor}\label{thm:simple groups(iii)}
Let $S$ be a finite nonabelian simple group and let $p$ be a prime. 
Suppose that $S$ has nontrivial cyclic Sylow $p$-subgroups. Write $N=S^t$ with $t \geq 2$.  
Suppose that $N\normal X$ with $p\nmid|X:N|$. 
Then $$k_{0,\sigma}(B_0(X))=k_{0,\sigma}(B_0(\norm X P))>2\sqrt{p-1}$$
where $P\in\Syl_p(X)$.
\end{cor}
\begin{proof} Let $P\in \syl p X$. By Proposition \ref{thm:simple groups ConjectureD} we have that
$$k_{0, \sigma}(B_0(X))=k_{0, \sigma}(B_0(\norm X P)).$$
Since $\norm X P$ has a normal Sylow $p$-subgroup, by Lemma \ref{lem:normalsylow}, we have that
$$k_{0, \sigma}(B_0(\norm X P))=k(\norm X P /\Phi(P){\bf O}_{p'}(\norm X P))\, .$$
Now, using that $t\geq 2$ we have that $P/\Phi(P)$ is not cyclic, which implies that 
$$k(\norm X P /\Phi(P){\bf O}_{p'}(\norm X P))>\lceil 2\sqrt{p-1}\rceil\,$$
by Theorem \ref{rank2} and this concludes the proof.
\end{proof}

\section{Proof of Theorem A}\label{sec:main}

We begin by proving the inequality part in Theorem A.

\begin{thm}\label{thm:lowerbound} Let $p$ be a prime and $G$ be a finite group of order divisible by $p$. Then
$k_{0, \sigma}(B_0(G))\geq 2\sqrt{p-1}$.
\end{thm}
\begin{proof} If $G$ has a normal Sylow $p$-subgroup $P$, then using Lemma \ref{lem:normalsylow} we have $k_{0,\sigma}(B_0(G))=k(G/\Phi(P)\oh{p'}{G})\geq 2\sqrt{p-1}$ by the main result of \cite{Mar16} and we are done in this case.

We proceed by induction on $|G|$. By Lemma \ref{lem:princblockabove}(ii) and the inductive hypothesis, we may assume that ${\bf O}_{p'}(G)=1$.
If $M\normal G$ and $p\mid |G:M|$ then by induction and Lemma \ref{lem:princblockabove}(i) we have $k_{0,\sigma}(B_0(G))\geq k_{0,\sigma}(B_0(G/M))\geq 2\sqrt{p-1}$ and we are done. 
Thus we may assume that $G$ has a unique minimal normal subgroup $N$, which is semisimple, of order divisible by $p$ and $p\nmid |G:N|$.

Say $N\cong S^t$, where $S$ is nonabelian simple of order divisible by $p$. Suppose that $S$ has cyclic Sylow $p$-subgroups. If $t=1$, then the Sylow $p$-subgroups of $G$ are also cyclic and we are done by Lemma \ref{lem:cyclic bijection galois}.
If $t\geq 2$ then we apply Theorem \ref{thm:simple groups}(iii).

Therefore, we may assume that the Sylow $p$-subgroups of $S$ are not cyclic. Let $k$ be the number of $\Aut(S)$-orbits on $\Irr_{p',\sigma}(B_0(S))$. If $t\geq 2$, since $G\leq\Aut(S)\wr\mathfrak{S}_t$, then $G$ has at least 
$\binom{k+t-1}{t}$
orbits on $\Irr_{p',\sigma}(B_0(N))$. By Theorem \ref{thm:simple groups}(ii)(b) we have that $k\geq 2(p-1)^{1/4}$. In particular $G$ has at least 
$$\frac{k(k+1)} 2=2\sqrt{p-1}+(p-1)^{1/4}\geq 2\sqrt{p-1}+1>\lceil 2\sqrt{p-1}\rceil$$
orbits on $\Irr_{p',\sigma}(B_0(N))$.
Since $|G:N|$ is not divisible by $p$, every such orbit lies under a distinct $\chi\in\Irr_{p',\sigma}(B_0(G))$ so we are done using Corollary \ref{cor:p'abelianquotsigmafixed}. If $t=1$, then the result follows from Theorem \ref{thm:simple groups}(ii)(a).
\end{proof}

\begin{thm}
Assume $k_{0, \sigma}(B_0(G))=\lceil 2\sqrt{p-1} \rceil$. Then the Sylow $p$-subgroups of $G$ are cyclic.
\end{thm}

\begin{proof} Let $G$ be a minimal counterexample to the statement. In particular $P \in \syl p G$ is not cyclic. 

\medskip

\textit{Step 1: We have $p\geq 5$.}

\medskip

If $p\leq 3$, then $\lceil 2\sqrt{p-1} \rceil=p$, so $k_{0, \sigma}(B_0(G))=\lceil 2\sqrt{p-1}\rceil$ implies that $P$ is cyclic by the main result of \cite{Riz-Sch-Val20}, contradicting the choice of $G$ as a minimal counterexample.

\medskip

\textit{Step 2: If $N$ is a minimal normal subgroup of $G$, then $p$ divides $|N|$ and $p$ does not divide $|G:N|$.}

\medskip

Assume otherwise. If $p\nmid |N|$ then $\Irr(B_0(G))=\Irr(B_0(G/N))$ by Lemma \ref{lem:princblockabove}(ii). By the minimality of $G$ we have that $G/N$ has cyclic Sylow $p$-subgroups, then so does $G$, a contradiction.

If $p$ divides $|G:N|$ then by Theorem \ref{thm:lowerbound} we have $$\lceil 2\sqrt{p-1}\rceil \leq \Irrg{B_0(G/N)}\leq \Irrg{B_0(G)}=\lceil 2\sqrt{p-1}\rceil.$$ Thus $G/N$ has cyclic Sylow $p$-subgroups and every $p'$-degree almost $p$-rational character of $B_0(G)$ lies over $1_N$ (and in the principal block of $G/N$).

Suppose that $N$ is a $p$-group. Let $P\in\Syl_p(G)$ and notice that $P$ is not cyclic but $P/N$ is cyclic and nontrivial. We have $N\cap\Phi(P)<\Phi(P)$. Let $\lambda\in\Irr(P/\Phi(P))$ be such that $N$ is not contained in $\ker\lambda$. Then $\lambda$ is $\sigma$-invariant and linear. Write $P\cent G P=P\times X$ and $\hat\lambda=\lambda\times 1_X\in\Irrg{B_0(P\cent G P)}$. By Lemma \ref{lem:Cargument} there is some $\chi\in\Irr _{p',\sigma}(B_0(G))$ contained in $\hat\lambda^G$. By Frobenius reciprocity, $\chi_{P\cent G P}$ contains $\hat\lambda$ so $N$ is not contained in $\ker\chi$, a contradiction. 

Therefore, we may assume that $N$ is a direct product of copies of a nonabelian simple group $S$ of order divisible by $p$. Write $N=S_1\times\dots\times S_t$, where $S_i\cong S$ for every $i$. By Theorem \ref{thm:lowerbound} and the hypothesis that $k_{0,\sigma}(B_0(G))=\lceil 2\sqrt{p-1}\rceil$, we have that $$k_{0, \sigma}(B_0(G/N))=\lceil 2\sqrt{p-1} \rceil$$ so it suffices to show that there is some $\chi\in\Irrg{B_0(G)}$ such that $N$ is not contained in $\ker\chi$. 

Write $M=PN$ and $H=M\cent G P$. By Theorem \ref{thm:simple groups}(i) and Lemma \ref{lem:G=PN} there is some $\psi\in\Irrg{B_0(M)}$ not containing $N$ in its kernel. By Alperin--Dade (Lemma \ref{lem:alpdade}), there is some extension $\tilde\psi\in\Irr(B_0(H))$ of $\psi$, and by Lemma \ref{lem:p'abelianquotsigmafixed} we have $\tilde\psi\in\Irrg{B_0(H)}$. By Lemma \ref{lem:Cargument}, the induced character $(\tilde\psi)^G$ contains some $\chi\in\Irrg{B_0(G)}$, and since $\chi_N$ contains $\theta$, $N$ is not contained in $\ker{\chi}$. This contradicts our assumption, so we have $p\nmid|G:N|$.

\medskip

\textit{Final step.} Let $N$ be a minimal normal subgroup of $G$.
By Step 2 we have $|G:N|$ is not divisible by $p$. If $N$ is a $p$-group, then $G$ has a normal Sylow $p$-subgroup and by Lemma \ref{lem:normalsylow} $k_{0,\sigma}(B_0(G))=k(G/\Phi(N)\oh{p'}G)$ and it follows from Theorem \ref{rank2} that $G$ is not a counterexample. 
Thus, we have that $N$ is semisimple, and $N$ is the only minimal normal subgroup of $G$. Furthermore, by using Lemma \ref{lem:p'abelianquotsigmafixed}, whenever $\theta\in\Irrg{B_0(N)}$ there is some $\chi\in\Irrg{B_0(G)}$ lying over $\theta$.
 If $N=S_1\times\dots\times S_t$ with $t\geq 2$, then arguing as in Theorem \ref{thm:lowerbound} we obtain that $k_{0, \sigma}(B_0(G))>\lceil2\sqrt{p-1}\rceil$, so we may assume $N$ is simple and $G$ is almost simple. If $N$ has cyclic Sylow $p$-subgroups, then so does $G$, contradicting the choice of $G$ as a counterexample. 
 Otherwise, Theorem \ref{thm:simple groups}(ii) shows $G$ is not a counterexample.
\end{proof}

\section{Final comments}\label{sec:final}

\subsection{Related problems}

Fix a prime number $p$ and write $$\mathcal S_p =\{ e+ \frac {p-1} e \ | \ e \text{ divides } p-1 \}.$$ In \cite{Hun-Mal-Mar22}, Hung, Malle and the first-named author of this article ask the following question. 
\begin{que}[Question 1.5 of \cite{Hun-Mal-Mar22}]
Is it true that the Sylow $p$-subgroups of a group $G$ of order divisible by $p$ are cyclic if 
$k_{0, \sigma}(B_0(G))\in \mathcal S _p$?
\end{que}

Notice that when $p-1$ is a perfect square, then $\lceil 2 \sqrt{p-1}\rceil=2\sqrt{p-1}= {\rm min}(\mathcal S_p)$. Then our main theorem provides a partial positive answer to the above question. It is still unknown whether there are finitely many primes $p$ satisfying  that $p-1$ is a perfect square. In fact, this is one of the 4 problems on prime numbers that Landau presented in his talk at the International Congress of Mathematicians in 1912.

In \cite{Cin-Kel23} the authors put forward the following conjecture.

\begin{conj}[C{\i}narc{\i}--Keller] Let $p$ be a prime. Let $a$ and $b$ be positive integers such
that $p -1 = ab$ and such that $|a - b|$ is minimal.
Then, for any finite group $G$ of order divisible by $p$, we have $k(G)\geq a+b$ and
 $k(G)=a+b$ if, and only if, $G\cong \mathsf{C}_p\rtimes \mathsf{C}_a$ or $\mathsf{C}_p\rtimes \mathsf{C}_b$.
 \end{conj}
 
 It is straightforward to see that $a+b$ is precisely the minimum of $\mathcal{S}_p$, where $\mathcal S_p$ is defined as above. Notice that $\lceil2\sqrt{p-1}\rceil \leq \min \mathcal S_p$, so our Theorem  \ref{rank2} provides evidence for the C{\i}narc{\i}--Keller conjecture in the case where $\lceil2\sqrt{p-1}\rceil = \min \mathcal S_p$ or, equivalently, when $\lceil2\sqrt{p-1}\rceil \in \mathcal S_p$. This happens significantly more often than $\sqrt{p-1}\in\mathbb{Z}$ (a quick check in \cite{GAP} shows that $\lceil2\sqrt{p-1}\rceil\in\mathcal{S}_p$ for $77$ out of the $168$ primes smaller than $1000$, whereas there are only $10$ primes smaller than $1000$ with $\sqrt{p-1}\in\mathbb{Z}$).
 
\subsection{Arbitrary blocks}\label{sec:AMN}

Recall that the H\'ethelyi-K\"ulshammer conjecture predicts that 
$$k(B)\geq 2\sqrt{p-1}$$
whenever $B$ is a $p$-block with nontrivial defect of some finite group $G$.
The inequality holds for $p\leq 3$ by the weak block orthogonality relation. As already mentioned in the Introduction, it was recently settled for principal blocks in \cite{Hun-Sch23}. 
However, outside these cases, little is known about this conjecture. In particular, it is not known whether it follows from the Alperin-McKay conjecture nor whether it holds for solvable groups.

 The following problem appeared recently in \cite{Nav23}.

\begin{prob}[Problem 2.5 of \cite{Nav23}]\label{prob:am implies hk} 
Let $V$ be a finite dimensional $\mathbb{F}_pH$-module, where $H$ is a finite $p'$-group and $\mathbb{F}_p$ is the field of $p$ elements. Let $G=VH$ be the semidirect product. Let $Z=\oh{p'}G=\cent H V$ and assume $Z\sbs \zent G$. Let $\lambda\in\Irr(Z)$. Then $|\Irr(G|\lambda)|\geq2\sqrt{p-1}$.
\end{prob}

As mentioned before its statement in \cite{Nav23}, if Problem \ref{prob:am implies hk} has a positive answer, then the blocks $B$ satisfying the Alperin--McKay conjecture also satisfy the H\' ethelyi--K\"ulshammer conjecture. Recall that the Alperin-McKay conjecture holds for a block $B$ if $k_0(B)=k_0(b)$, where $b$ is the Brauer first main correspondent of $B$.
If instead we assume that the block $B$ satisfies the Alperin--McKay--Navarro conjecture \cite[Conjecture]{Nav04}, then $B$ would satisfy the blockwise version of our Theorem \ref{thm:maintheorem}. We provide a proof of this fact in Proposition \ref{prop:AMN} below, but let us first recall the statement of the Alperin--McKay--Navarro conjecture.

For a fixed prime $p$, let $\mathcal{H}\leq \Gal(\mathbb{Q}_{|G|}/\mathbb{Q})$ be the subgroup generated by the field automorphisms $\tau$ which send $p'$-roots of unity $\xi$ to $\xi^{p^e}$ for some integer $e$. For a $p$-block $B$ we let $\Irr_0(B)$ denote the set of characters of $B$ with height zero. The group $\mathcal H$ acts on the set $\{ \irr B \ | \ B \text{ $p$-block of $G$}\}$, and we denote by $\mathcal H_B$ the stabilizer of $\irr B$ under the action of $\mathcal H$.

\begin{conj}[Alperin--McKay--Navarro]\label{conj:AMN}
Let $B$ be a $p$-block of $G$ with defect group $D$ and let $b$ be its Brauer correspondent block in $\norm G D$. Then there is an $\mathcal{H}_B$-equivariant bijection
$\Irr_0(B)\rightarrow\Irr_0(b)$.
\end{conj}

Notice that our field automorphism $\sigma$ from the introduction is an element of $\mathcal{H}$. 
Moreover the action of $\langle \sigma \rangle$ stabilizes $\irr B$ for every $p$-block $B$ (because Brauer characters are $\sigma$-invariant).
In particular, Conjecture \ref{conj:AMN} predicts that the action of $\langle \sigma\rangle$ on characters fixes the same number of height zero characters in $B$ as it does in its Brauer correspondent block $b$, that is
$$k_{0, \sigma}(B)=k_{0, \sigma}(b)\, .$$

\begin{pro}\label{prop:AMN}
Let $G$ be a finite group of order divisible by $p$, $B$ a $p$-block of $G$ with nontrivial defect group $D$. Assume Problem \ref{prob:am implies hk} has a positive answer for all finite groups and that the Alperin--McKay--Navarro conjecture holds for $B$. Then $$k_{0,\sigma}(B)\geq 2\sqrt{p-1}\, .$$
\end{pro}
\begin{proof}
Let $b$ be the Brauer correspondent block of $B$ in $\norm G D$. We have that $k_{0,\sigma}(B)=k_{0,\sigma}(b)$ by the Alperin--McKay--Navarro conjecture. By \cite[Lemma 2.2]{Val23} we have that
$$k_{0, \sigma}(b)=k(\overline b)$$
where $\overline b$ is a block of $\norm G D/\Phi(D)$ with defect group $D/\Phi(D)$. 
By \cite[Theorem 6]{Rey63} applied to $\overline{b}$ there is a block $b'$ of a finite group $K$ with normal defect group $V\in\Syl_p(K)$, 
$V\cong D/\Phi(D)$ and with a height-preserving bijection $\Irr(\overline{b})\rightarrow\Irr(b')$. It follows that
$k_{0,\sigma}(B)=k(b')$. By Schur--Zassenhaus, the group $V$ has a complement $H$ in $K$, so $K=VH$ with $V$ an $\mathbb{F}_pH$-module.
By Fong's theorem \cite[Theorem 10.20]{Nav98}, there is some $K$-invariant $\lambda\in\Irr(\oh{p'}{K})$ with $\Irr(b')=\Irr(K|\lambda)$. 
By character triple isomorphisms \cite[Problem 8.13]{Nav98} we may assume $\oh{p'}K\sbs\zent K$. Using the assumption that Problem \ref{prob:am implies hk} has a positive answer, we have
$$k_{0,\sigma}(B)=k(b')=|\Irr(K|\lambda)|\geq2\sqrt{p-1}$$
as desired.
\end{proof}

The version for principal blocks of Problem \ref{prob:am implies hk} is the case where $Z=1$, and a positive answer in this case is guaranteed by \cite{Mar16}. Arguing as above, we get that if the Alperin--McKay--Navarro conjecture holds for $B_0(G)$ then $$k_{0,\sigma}(B_0(G))=k(\norm G P/\Phi(P)\oh{p'}{\norm G P})$$ and the equality part of Theorem \ref{thm:maintheorem} follows from Theorem \ref{rank2}. 
 
 \smallskip
 
We close our comments with a last observation. Notice that for every group $G$ and every prime $p$ we have that $k_{0,\sigma}(B_0(G))\geq 1$ because the principal $p$-block contains the trivial character $1_G$. However, the existence of almost $p$-rational characters of height zero in arbitrary $p$-blocks had not been observed until quite recently (see \cite{Nav-Riz}). Now that we know that $k_{0,\sigma}(B)\geq 1$ for arbitrary blocks and primes, we can guarantee that the lower bound
$$k_{0, \sigma}(B)\geq 2\sqrt{p-1}$$
holds whenever $p\leq 3$ by a direct application of \cite[Lemma 1.4]{Riz-Sch-Val20}.

\bibliographystyle{alpha}

\small{
\vspace{1cm}

(A. Mar\'oti) {\sc{Alfr\'ed R\'enyi Institute of Mathematics, H-1053 Budapest, Hungary}}

\textit{Email address:} \href{mailto:maroti.attila@renyi.hu}{\texttt{maroti.attila@renyi.hu}}

\medskip

(J. M. Mart\'inez) {\sc{Dipartimento di Matematica e Informatica ‘Ulisse Dini’, 50134 Firenze, Italy}}

\textit{Email address:} \href{mailto:josepmiquel.martinezmarin@unifi.it}{\texttt{josepmiquel.martinezmarin@unifi.it}}

\medskip

(A. A. Schaeffer Fry) {\sc{ (primary) Dept. Mathematics - University of Denver, Denver, CO 80210, USA; and 
Dept. Mathematics and Statistics - MSU Denver, Denver, CO 80217, USA}}

\textit{Email address:} \href{mailto:mandi.schaefferfry@du.edu}{\texttt{mandi.schaefferfry@du.edu}}

\medskip

(C. Vallejo) {\sc{Dipartimento di Matematica e Informatica ‘Ulisse Dini’, 50134 Firenze, Italy}}

\textit{Email address:} \href{mailto:arolina.vallejorodriguez@unifi.it}{\texttt{carolina.vallejorodriguez@unifi.it}}

\end{document}